\documentclass[12pt]{amsart}

\usepackage{amsmath,bm,amssymb,amsthm,textcomp}
\usepackage{enumitem}
\usepackage{mathtools}
\usepackage{hyperref}
\usepackage[numbers,sort&compress]{natbib}
\usepackage{tikz-cd}
\usetikzlibrary{cd}
\DeclareMathOperator{\admissible}{ad}
\DeclareMathOperator{\diam}{diam}
\DeclareMathOperator{\len}{len}
\DeclareMathOperator{\hdim}{\dim_H}

\theoremstyle{plain}
\newtheorem{theorem}{Theorem}[section]
\newtheorem{lemma}[theorem]{Lemma}
\newtheorem{proposition}[theorem]{Proposition}
\newtheorem{corollary}[theorem]{Corollary}
\newtheorem*{claim}{Claim}

\theoremstyle{definition}

\theoremstyle{remark}
\newtheorem{remark}[theorem]{Remark}
\newtheorem{remarks}[theorem]{Remarks}

\counterwithin{equation}{section}

\allowdisplaybreaks

\begin{document}

\title{Hausdorff dimensions in Pierce expansions}

\author{Min Woong Ahn}
\address{Department of Mathematics, SUNY at Buffalo, Buffalo, NY 14260-2900, USA}
\email{minwoong@buffalo.edu}

\date{\today}


\begin{abstract}
The digits of the Pierce expansion obey the law of large numbers, the central limit theorem, and the law of the iterated logarithm in the Lebesgue measure sense. We calculate the Hausdorff dimensions of the exceptional sets associated with each of the three laws. We further determine the Hausdorff dimensions of certain sets arising in Pierce expansions.
\end{abstract}

\maketitle

\tableofcontents

\section{Introduction} \label{introduction}

Given $x \in (0,1]$, we define
\begin{align} \label{Pierce algorithm 1}
d_1(x) \coloneqq \left\lfloor \frac{1}{x} \right\rfloor
\quad \text{and} \quad 
T(x) \coloneqq 1 - d_1 (x) x,
\end{align}
where $\lfloor y \rfloor$ denotes the largest integer not exceeding $y \in \mathbb{R}$. A sequence $(d_k(x))_{k \in \mathbb{N}}$ associated with $x$ is recursively defined by
\begin{align} \label{Pierce algorithm 2}
d_k(x) \coloneqq d_1 (T^{k-1}(x)) \quad \text{provided that } T^{k-1}(x) \neq 0, 
\end{align}
for each $k \in \mathbb{N}$. We denote by $\langle d_1(x), d_2(x), \dots\rangle_P$ the {\em Pierce expansion} (or {\em alternating Engel expansion}) of $x$, i.e., 
\begin{align} \label{Pierce expansion}
x = \langle d_1(x), d_2(x), \dotsc \rangle_P \coloneqq \sum_{j=1}^\infty \left( (-1)^{j+1} \prod_{k=1}^j \frac{1}{d_k(x)} \right),
\end{align}
where $( d_k(x) )_{k \in \mathbb{N}}$, called the sequence of Pierce expansion {\em digits}, is defined by \eqref{Pierce algorithm 1} and \eqref{Pierce algorithm 2}. If $n \coloneqq \inf \{ k \in \mathbb{N} : T^k(x) = 0 \}$ is infinite, then $(d_k(x))_{k \in \mathbb{N}}$ is in $\mathbb{N}^\mathbb{N}$ and $d_k(x) < d_{k+1}(x)$ for all $k \in \mathbb{N}$. Otherwise, $d_k(x)$ is not defined for all $k \geq n+1$, so that \eqref{Pierce expansion} reduces to a sum of finitely many terms, in which case we write $x = \langle d_1(x), d_2(x), \dotsc, d_n(x) \rangle_P$. If, further, $n \neq 1$, then $d_{n-1}(x) + 1 < d_n(x)$ (see \cite[Proposition 2.2]{Ahn23} and \cite[pp.~23--24]{Sha86}). 

It is a classical well-known fact that every irrational number in $(0,1]$ has an infinite length sequence of Pierce digits, while every rational number has a finite length sequence of Pierce digits. Moreover, there is a one-to-one correspondence via the Pierce expansion between the irrationals in $(0,1]$ and the collection of all strictly increasing infinite sequences of positive integers. We refer the reader to \cite{Ahn23, DDF22, ES91, Fan15, PVB98, Pie29, Sha84, Sha86, Sha93, Var17, VBP99} or \cite[Chapter 2]{Sch95} for the arithmetic and metric properties and applications of Pierce expansions.

In 1986, Shallit \cite{Sha86} examined the limiting behavior of the Pierce expansion digits and obtained the following results:
\begin{enumerate}[label=\upshape\arabic*., leftmargin=*, widest=3]
\item
{\em The law of large numbers (LLN).} (p.~36, Theorem 16) For Lebesgue-almost every $x \in (0,1]$,
\begin{align} \label{law of large numbers}
\lim_{n \to \infty} (d_n(x))^{1/n} = e.
\end{align}
\item
{\em The central limit theorem (CLT).} (p.~36, Theorem 17) For any $t \in \mathbb{R}$,
\begin{align} \label{central limit theorem}
\lim_{n \to \infty} \lambda \left( \frac{\log d_n(x) - n}{\sqrt{n}} < t \right) = \Phi (t),
\end{align}
where $\lambda$ denotes the Lebesgue measure on $(0,1]$ and $\Phi \colon \mathbb{R} \to \mathbb{R}$ is the normal distribution given by $\Phi (t) \coloneqq ({1}/{\sqrt{2\pi}}) \int_{-\infty}^t e^{-u^2/2} \, du$.
\item
{\em The law of the iterated logarithm (LIL).} (p.~37, Theorem 18) For Lebesgue-almost every $x \in (0,1]$,
\begin{align} \label{law of the iterated logarithm}
\limsup_{n \to \infty} \frac{\log d_n(x) - n}{\sqrt{2n \log \log n}} = 1
\quad \text{and} \quad
\liminf_{n \to \infty} \frac{\log d_n(x) - n}{\sqrt{2n \log \log n}} = -1.
\end{align}
\end{enumerate}

We remark that whenever we investigate the limiting behaviors of the digits, e.g., the above three laws, we confine ourselves to irrational numbers only. This is because, as discussed in the second paragraph, any rational number $x \in (0,1]$ has the digit sequence of finite length, i.e., $d_n(x)$ is not defined from some point onwards. Since rational numbers are countable, such a restriction does not affect the Hausdorff dimension of any set in $(0,1]$.

A brief mention of the Engel expansion, which can be seen as a non-alternating version of the Pierce expansion, is in order. In 1958, Erd{\H o}s, R{\'e}nyi, and Sz{\"u}sz \cite{ERS58} established the same result as \eqref{law of large numbers} in the context of Engel expansions. To be precise, let $\langle \delta_1(x), \delta_2(x), \dotsc \rangle_E$ denote the Engel expansion of $x \in (0,1]$, i.e.,
\[
x = \langle \delta_1(x), \delta_2(x), \dotsc \rangle_E \coloneqq \sum_{j=1}^\infty \left( \prod_{k=1}^j \frac{1}{\delta_k(x)} \right),
\]
where $( \delta_k (x) )_{k \in \mathbb{N}}$ is a sequence of positive integers satisfying $2 \leq \delta_{k}(x) \leq \delta_{k+1}(x)$ for each $k \in \mathbb{N}$ (see \cite[p.~17]{Gal76}). The authors of \cite{ERS58} proved the LLN of Engel expansion digits: {\em For Lebesgue-almost every $x \in (0,1]$,}
\begin{align} \label{Engel convergence law}
\lim_{n \to \infty} ( \delta_n (x) )^{1/n} = e .
\end{align}
The CLT and the LIL in terms of Engel expansion digits, analogous to \eqref{central limit theorem} and \eqref{law of the iterated logarithm}, respectively, were also proved by the authors.

Regarding the size of certain subsets of $(0,1]$ arising in the Engel expansion, Galambos brought up the following three questions \cite[p.~132, Problem 9]{Gal76}:
\begin{enumerate}[label=\upshape Q\arabic*., ref=Q\arabic*, leftmargin=*, widest=3]
\item \label{Q1}
What is the Hausdorff dimension of the set where \eqref{Engel convergence law} fails?
\item \label{Q2}
For each $k \in \mathbb{N}$, let
\[
A^{(E)}_k \coloneqq \{ x \in (0,1] : \log \delta_n(x) \geq kn \text{ for all } n \in \mathbb{N} \}.
\]
What is the Hausdorff dimension of the set $A^{(E)}_k$?
\item \label{Q3}
For each $k \in \mathbb{N}$, let
\[
B^{(E)}_k \coloneqq \left\{ x \in (0,1] : \frac{\delta_{n+1}(x)}{\delta_n(x)} \leq k \text{ for all } n \in \mathbb{N} \right\}.
\]
What is the Hausdorff dimension of the set $B^{(E)}_k$?
\end{enumerate}
The first question, evidently, is directly concerned with the LLN, and the other two questions are also related to some of the three mentioned limit theorems. See Remarks \ref{remark B k}(\ref{remark B k 2}) and \ref{remark to corollary LIL exceptional set}, where we discuss these relations in the context of Pierce expansions.

The first answer to these problems became known twenty-four years after the problems were posed. For \ref{Q1}, in 2000, Jun Wu \cite{Wu00} proved that the set where \eqref{Engel convergence law} fails is of full Hausdorff dimension by an ingenious construction. Specifically, for some fixed subset $A \subseteq (0,1]$ with positive Lebesgue measure (hence with full Hausdorff dimension), he defined maps $g_n \colon A \to (0,1]$ for all large enough $n \in \mathbb{N}$ such that \eqref{Engel convergence law} fails on $g_n(A)$ and $g_n^{-1} \colon g_n(A) \to A$ is $(1/\alpha_n)$-H{\"o}lder continuous, where $\alpha_n \to 1^+$ as $n \to \infty$, in which case the Hausdorff dimension of $g_n(A)$ can be made arbitrarily close to $1$. One year later, Liu and Jun Wu \cite{LW01} formulated a more general statement than what is required to answer \ref{Q1} and \ref{Q2}: {\em For each $\alpha \in [1, \infty)$, let
\[
A^{(E)}(\alpha) \coloneqq \left\{ x \in (0,1] : \lim_{n \to \infty} ( \delta_n(x) )^{1/n} = \alpha \right\}.
\]
Then $\hdim A^{(E)}(\alpha) = 1$ for each $\alpha \in [1, \infty)$, where $\hdim$ denotes the Hausdorff dimension.} As corollaries, they showed that the Hausdorff dimensions in \ref{Q1} and \ref{Q2} both equal $1$. For \ref{Q3}, Wang and Jun Wu \cite{WW07} proved that $B_k^{(E)}$ is also of full Hausdorff dimension for each $k \in \mathbb{N}$ with $k \geq 2$. Since then, a number of studies concerning the Hausdorff dimension of certain sets defined in terms of Engel expansion digits have been conducted. See, e.g., \cite{FW18, LW03, LL18, SW20, SW21, She10, Wu03} for further results. Recently, Shang and Min Wu \cite{SW21} determined the Hausdorff dimension of the set
\[
E^{(E)} (\phi) \coloneqq \left\{ x \in (0,1] : \lim_{n \to \infty} \frac{\log \delta_n(x)}{\phi (n)} = 1 \right\},
\]
where $\phi \colon \mathbb{N} \to (0, \infty)$ is a non-decreasing function satisfying $\phi (n) / \log n \to \infty$ as $n \to \infty$.

Returning to the main topic, it has been thirty-eight years since Shallit \cite{Sha86} established three laws \eqref{law of large numbers}--\eqref{law of the iterated logarithm} which the infinite digit sequences of Pierce expansions obey. One natural question is to determine the Hausdorff dimension of certain sets defined in terms of Pierce expansion digits. However, unlike the Engel expansion, its alternating counterpart has not been receiving much attention in the field of fractal studies.

It is worth pointing out that the study of Hausdorff dimension in the Pierce expansion context is interesting in its own right. Engel expansions and Pierce expansions have similar but distinct metric properties. Based on the similarities, it seems reasonable to expect that a set defined in terms of Pierce expansion digits and its Engel expansion counterpart have the same Hausdorff dimension.

In this paper, we determine the Hausdorff dimensions of certain subsets of the half-open unit interval $(0,1]$ defined in terms of Pierce expansion digits, including the sets originating from the LLN, the CLT, and the LIL. In the process, the three problems of Galambos for Pierce expansions will be completely solved. The main results are as follows. For reference, we list the papers in which analogous results for Engel expansions are present, wherever possible.

The first main result is concerned with the growth speed of the digits and the LLN. (See \cite[Theorems 3.2 and 3.6]{SW20} and \cite[Theorem 3.1]{SW21} for analogous results for the Engel expansion.)

\begin{theorem} \label{theorem E phi} 
Let $\phi \colon \mathbb{N} \to (0, \infty)$ be a non-decreasing function such that $\phi (n)/\log n \to \gamma$ as $n \to \infty$ for some $\gamma \in [0, \infty]$. Define 
\begin{align} \label{definition of E phi}
E(\phi) \coloneqq \left\{ x \in (0,1]: \lim_{n \to \infty} \frac{\log d_n(x)}{\phi (n)} = 1 \right\}.
\end{align}
Then
\[
\hdim E(\phi) = \max \left\{ 0, \frac{1-\gamma^{-1}}{1+\xi} \right\},
\quad \text{where} \quad \xi \coloneqq \limsup_{n \to \infty} \frac{\phi (n+1)}{\sum \limits_{k=1}^n \phi (k)},
\]
with the conventions $\infty^{-1} = 0$ and $0^{-1} = \infty$. More precisely,
\[
\hdim E(\phi) =
\begin{cases} 
(1+\xi)^{-1}, &\text{if } \gamma = \infty; \\
1 - \gamma^{-1}, &\text{if } \gamma \in [1, \infty); \\
0, &\text{if } \gamma \in [0, 1), \text{ because } E(\phi) = \varnothing.
\end{cases}
\]
\end{theorem}

\begin{remarks}
\begin{enumerate}[label=\upshape(\arabic*), ref=\arabic*, leftmargin=*, widest=2]
\item
It suffices to assume only $\limsup_{n \to \infty} \phi (n)/ \log n < 1$ to have $E(\phi) = \varnothing$; see Lemma \ref{E phi lemma 1}.

\item
When $\gamma = 1$, although $E(\phi)$ is of null Hausdorff dimension, it is not necessarily empty. For instance, let $\phi \colon \mathbb{N} \to (0, \infty)$ be a non-decreasing function defined by $\phi (n) \coloneqq \log n$ for each $n \in \mathbb{N}$. By Proposition \ref{strict increasing condition}, there exists $x \in (0,1]$ such that $d_n(x) = n$ for all $n \in \mathbb{N}$. Then $x \in E(\phi)$. Similarly, when $\gamma = \xi = \infty$, the set $E(\phi)$ has Hausdorff dimension zero but it is not necessarily empty.
\end{enumerate}
\end{remarks}

The LLN states that $(d_n (x))^{1/n} \to e$ as $n \to \infty$ for Lebesgue-almost every $x \in (0,1]$. The following corollary to Theorem \ref{theorem E phi} tells us that for any $\alpha \in [1, \infty]$, the set on which $(d_n (x))^{1/n} \to \alpha$ as $n \to \infty$ is also quite large in the Hausdorff dimension sense. (See \cite[Theorem]{LW01} for an analogous result for the Engel expansion.)

\begin{corollary} \label{corollary A alpha}
For each $\alpha \in [-\infty, \infty]$, let
\begin{align} \label{definition of A alpha}
A(\alpha) \coloneqq \left\{ x \in (0,1]: \lim_{n \to \infty} (d_n (x))^{1/n} = \alpha \right\}.
\end{align}
Then,
\[
\hdim A(\alpha) =
\begin{cases}
1, &\text{if } \alpha \in [1,\infty]; \\
0, &\text{if } \alpha \in [-\infty,1), \text{ because } A(\alpha) = \varnothing.
\end{cases}
\]
\end{corollary}

By Corollary \ref{corollary A alpha}, we can immediately answer \ref{Q1} in the context of Pierce expansions. (See \cite[Corollary 3]{LW01} and \cite[Theorem]{Wu00} for an analogous result for the Engel expansion.)

\begin{corollary} \label{corollary LLN exceptional set}
The Hausdorff dimension of the set where \eqref{law of large numbers} fails is $1$, i.e.,
\[
\hdim \left\{ x \in (0,1] : (d_n (x))^{1/n} \not \to e \text{ as } n \to \infty \right\} = 1.
\]
\end{corollary}

Historically, for the Engel expansion, an analogous version of the following corollary was proved earlier than the analogues of Theorem \ref{theorem E phi}. Notice that Corollary \ref{corollary E phi} has more restrictive conditions on the map $\phi$ than the theorem does. (See \cite[Theorem 1.1]{LL18} for an analogous result for the Engel expansion.)

\begin{corollary} \label{corollary E phi}
Let $\phi \colon \mathbb{N} \to (0, \infty)$ be a non-decreasing function such that $\phi (n+1) - \phi (n) \to \infty$ as $n \to \infty$. Let $E(\phi)$ and $\xi$ be defined as in Theorem \ref{theorem E phi}. Then
\[
\hdim E(\phi) = (1+\xi)^{-1}.
\]
\end{corollary}

The second main result is the answer to \ref{Q2} in a more general setting.

\begin{theorem} \label{theorem A kappa}
For each $\kappa \in (-\infty, \infty)$, let
\begin{align} \label{definition of A kappa}
A_\kappa \coloneqq \{ x \in (0,1] : \log d_n(x) \geq \kappa n \text{ for all } n \in \mathbb{N} \}.
\end{align}
Then, for each $\kappa \in (-\infty, \infty)$,
\[
\hdim A_\kappa = 1.
\]
\end{theorem}

\begin{remark} \label{remark A k}
Notice that the Pierce expansion version of \ref{Q2} simply deals with a special case of the preceding theorem, in the sense that $\mathbb{N} \subseteq (-\infty, \infty)$. Hence, in answer to the question, each $A_k$, $k \in \mathbb{N}$, has full Hausdorff dimension. (See \cite[Corollary 1]{LW01} for an analogous result for the Engel expansion.)
\end{remark}

The third main result is concerned with the ratio between two consecutive digits. (See \cite[Corollary 2]{LW01} for an analogous result for the Engel expansion.)

\begin{theorem} \label{theorem B alpha}
For each $\alpha \in [-\infty, \infty]$, let
\begin{align*} 
B(\alpha) \coloneqq \left\{ x \in (0,1]: \lim_{n \to \infty} \frac{d_{n+1}(x)}{d_n(x)} = \alpha \right\}.
\end{align*}
Then
\[
\hdim B(\alpha) =
\begin{cases}
1, &\text{if } \alpha \in [1, \infty]; \\
0, &\text{if } \alpha \in [-\infty, 1), \text{ because } B(\alpha) = \varnothing.
\end{cases}
\]
\end{theorem}

\begin{remarks} \phantomsection \label{remark to theorem B alpha}
\begin{enumerate}[label=\upshape(\arabic*), ref=\arabic*, leftmargin=*, widest=2]
\item \label{remark to theorem B alpha 1}
In fact, Corollary \ref{corollary A alpha} can also be regarded as a corollary to Theorem \ref{theorem B alpha} (and not only to Theorem \ref{theorem E phi}), since $B(\alpha) \subseteq A(\alpha)$ for each $\alpha \in [1, \infty]$. See Remarks \ref{remark to the proof of theorem B alpha}(\ref{remark to the proof of theorem B alpha 1}) for the details.

\item \label{remark to theorem B alpha 2}
In \cite[Corollary 2]{LW01}, the authors considered the set
\[
\left\{ x \in (0,1] : \lim_{n \to \infty} \frac{\delta_{n+1}(x)}{\delta_n(x)-1} = \alpha \right\}
\]
and asserted that it has Hausdorff dimension $1$ if $\alpha \in [1, \infty)$. For each $\alpha \in [-\infty, \infty]$, let
\[
B'(\alpha) \coloneqq \left\{ x \in (0,1] : \lim_{n \to \infty} \frac{d_{n+1}(x)}{d_n(x)-1} = \alpha \right\}.
\]
For any irrational $x \in (0,1]$, we have $d_n(x) \geq n$ for each $n \in \mathbb{N}$ by Proposition \ref{digit condition lemma}(\ref{digit condition lemma 2}), so $d_n(x) \to \infty$ as $n \to \infty$. Thus, it is clear that $B(\alpha) = B'(\alpha)$ for any $\alpha \in [-\infty, \infty]$, and hence $\hdim B(\alpha) = \hdim B'(\alpha)$.
\end{enumerate}
\end{remarks}

The following corollary is a more general statement than the answer required for \ref{Q3}. 

\begin{corollary} \label{corollary B kappa} 
For each $\kappa \in (-\infty, \infty)$, let
\[
B_\kappa \coloneqq \left\{ x \in (0,1] : \frac{d_{n+1}(x)}{d_n(x)} \leq \kappa \text{ for all } n \in \mathbb{N} \right\}.
\]
Then
\[
\hdim B_\kappa = 
\begin{cases}
1, &\text{if } \kappa \in (1, \infty); \\
0, &\text{if } \kappa \in (-\infty, 1], \text{ because } B_\kappa = \varnothing.
\end{cases}
\]
\end{corollary}

\begin{remarks} \label{remark B k}
\begin{enumerate}[label=\upshape(\arabic*), ref=\arabic*, leftmargin=*, widest=2]
\item \label{remark B k 1}
The answer to \ref{Q3} in the context of Pierce expansions is immediate from the preceding corollary: each $B_k$, $k \in \mathbb{N} \setminus \{ 1 \}$, has Hausdorff dimension $1$, while $B_1$, being empty, has Hausdorff dimension $0$. (See \cite[Corollary 2.4]{WW07} for an analogous result for the Engel expansion.)

\item \label{remark B k 2}
The set $B_2$ is an exceptional set to the LLN, and this establishes a connection between the Pierce expansion version of \ref{Q3} and the LLN. Moreover, Corollary \ref{corollary B kappa} implies Corollary \ref{corollary LLN exceptional set} and solves \ref{Q1}. See Remark \ref{remark to proof of corollary B kappa} for the details.
\end{enumerate}
\end{remarks}

The fourth main result is concerned with the ratio between the logarithms of two consecutive digits. Compared to Theorem \ref{theorem B alpha}, the following theorem helps in understanding the numbers whose digits have a lot faster, e.g., doubly exponential growth rates. (See \cite[Theorem 1.1]{She10} and  \cite[Corollary 2.4]{WW06} for analogous results for the Engel expansion.)

\begin{theorem} \label{theorem ratio of logarithms}
For each $\alpha \in [-\infty, \infty]$, let
\begin{align} \label{definition of F alpha}
F(\alpha) \coloneqq \left\{ x \in (0,1] : \lim_{n \to \infty} \frac{\log d_{n+1}(x)}{\log d_n(x)} = \alpha \right\}.
\end{align}
Then
\[
\hdim F(\alpha) =
\begin{cases}
\alpha^{-1}, &\text{if } \alpha \in [1, \infty]; \\
0, &\text{if } \alpha \in [-\infty, 1), \text{ because } F(\alpha) = \varnothing,
\end{cases}
\]
with the convention $\infty^{-1} = 0$.
\end{theorem}

\begin{remarks}
\begin{enumerate}[label=\upshape(\arabic*), ref=\arabic*, leftmargin=*, widest=2]
\item
For any $x \in ( {2}^{-1}, 1 ]$, we have $d_1(x) = 1$, i.e., $\log d_1(x) = 0$, by definition of $d_1(x)$ in \eqref{Pierce algorithm 1}. But, for any irrational $x \in (0,1]$, we have $d_n(x) \geq n$ for all $n \in \mathbb{N}$ by Proposition \ref{digit condition lemma}(\ref{digit condition lemma 2}), so $\log d_n (x) > 0$ for all $n \in \mathbb{N} \setminus \{ 1 \}$. Thus, for any irrational $x \in (0,1]$, the quotient $\log d_{n+1}(x)/\log d_n(x)$ in \eqref{definition of F alpha} is finite for all integers $n \geq 2$.

\item
Although $F(\infty)$ is of null Hausdorff dimension, it is not empty. For instance, consider $x \in (0,1]$ with digits $d_n(x) = 2^{g(n)}$ where $g(n) \coloneqq 2^{2^n}$ for each $n \in \mathbb{N}$, whose existence is guaranteed by Proposition \ref{strict increasing condition}. Then $\log d_{n+1}(x) / \log d_n(x) = g(n) \to \infty$ as $n \to \infty$, and thus $x \in F(\infty)$.
\end{enumerate}
\end{remarks}

The fifth main result is concerned with the sets arising from the CLT and the LIL.

\begin{theorem} \label{theorem C psi beta}
Let $\psi \colon \mathbb{N} \to (0, \infty)$ be a function such that 
\begin{enumerate}[label=\upshape(\roman*), ref=\roman*, leftmargin=*, widest=ii]
\item \label{theorem C psi beta 1}
$\psi (n+1) - \psi (n) \to 0$ as $n \to \infty$,
\item \label{theorem C psi beta 2}
$e^{pn} \psi (n) \to \infty$ as $n \to \infty$ for any positive real number $p$.
\end{enumerate} 
For each $\beta \in [-\infty, \infty]$, let
\begin{align} \label{definition of C psi beta}
C(\psi, \beta) \coloneqq \left\{ x \in (0,1] : \lim_{n \to \infty} \frac{\log d_n(x) - n}{\psi (n)} = \beta \right\}.
\end{align}
Then, for each $\beta \in [-\infty, \infty]$,
\[
\hdim C(\psi, \beta) = 1.
\]
\end{theorem}

The CLT has nothing to do with the pointwise limit of the quotient $(\log d_n(x) - n)/{\sqrt{n}}$, but due to the theorem, a reasonable expected growth rate of $\log d_n(x) - n$ is $n^{1/2}$. The following theorem shows that there exist points with different growth rates. In fact, for any real exponent $\alpha$, including $1/2$, there are plenty of points in the Hausdorff dimension sense sharing a common pointwise limit. (See \cite[Theorems 1.1 and 1.2 and Corollary 2.3]{LW03} for analogous results for the Engel expansion.)

\begin{theorem} \label{theorem E alpha beta}
For each $\alpha \in (-\infty, \infty)$ and $\beta \in [-\infty, \infty]$, let
\begin{align*} 
E(\alpha, \beta) \coloneqq \left\{ x \in (0,1] : \lim_{n \to \infty} \frac{\log d_n(x) - n}{n^\alpha} = \beta \right\}.
\end{align*}
Then
\[
\hdim E(\alpha, \beta) = 
\begin{cases}
1, &\text{if } \begin{cases} \alpha \in (-\infty, 1); \\ \alpha = 1 \text{ and } \beta \in [-1,\infty]; \\ \alpha \in (1, \infty) \text{ and } \beta \in [0, \infty]; \end{cases}  \\
0, &\text{otherwise} , \text{ because } E(\alpha, \beta) = \varnothing.
\end{cases}
\]
\end{theorem}

The following result concerning the sets originating from the LIL is a direct consequence of Theorem \ref{theorem C psi beta}. (See \cite[Theorem 2.4]{LW03} for an analogous result for the Engel expansion.)

\begin{corollary} \label{corollary L beta}
For each $\beta \in [-\infty, \infty]$, let
\begin{align} \label{definition of L beta}
L(\beta) \coloneqq \left\{ x \in (0,1] : \lim_{n \to \infty} \frac{\log d_n(x) - n}{\sqrt{2n \log \log n}} = \beta \right\}.
\end{align}
Then, for each $\beta \in [-\infty, \infty]$
\[
\hdim L(\beta) = 1.
\]
\end{corollary}

The following corollary shows that the exceptional set on which \eqref{law of the iterated logarithm} fails is of full Hausdorff dimension, while the LIL tells us that the set of exceptions is of null Lebesgue measure. 

\begin{corollary} \label{corollary LIL exceptional set}
The Hausdorff dimension of the set where \eqref{law of the iterated logarithm} fails is $1$, i.e.,
\[
\hdim \left\{ x \in (0,1] : \limsup_{n \to \infty} \frac{\log d_n(x) - n}{\sqrt{2n \log \log n}} \neq 1
\text{ or }
\liminf_{n \to \infty} \frac{\log d_n(x) - n}{\sqrt{2n \log \log n}} \neq -1 \right\} = 1.
\]
\end{corollary}

\begin{remark} \label{remark to corollary LIL exceptional set}
Galambos' question \ref{Q2} in the context of Pierce expansions is related to the LIL as follows. Recall first that the question is concerned with the sets $A_\kappa$, $\kappa \in \mathbb{N}$, defined as in \eqref{definition of A kappa}. For each $\kappa \in (1, \infty)$, the set $A_\kappa$ is an exceptional set to the LIL. Furthermore, Corollary \ref{corollary LIL exceptional set} can be viewed as a consequence of Theorem \ref{theorem A kappa}, which answers the question. See Remark \ref{remark to proof of corollary LIL exceptional set} for the details.
\end{remark}

This paper is organized as follows. In Section \ref{preliminaries}, we introduce some basic notions and properties of Pierce expansions and useful calculation methods for the Hausdorff dimensions. In Section \ref{auxiliary results}, we derive a few auxiliary results on the Hausdorff dimensions of certain subsets of $(0,1]$. In Section \ref{proofs of the results}, we prove all the results mentioned in the introduction.

Throughout the paper, for a subset $F$ of $(0,1]$, we denote by $\diam F$ the diameter, $\hdim F$ the Hausdorff dimension, $\mathcal{H}^s(F)$ the $s$-dimensional Hausdorff measure, and $\overline{F}$ the closure. If $F$ is a subinterval of $(0,1]$, we write $\len F$ for its length. We write $\mathbb{I}$ for the set of irrational numbers in $(0,1]$, i.e., $\mathbb{I} \coloneqq (0,1] \setminus \mathbb{Q}$. We denote $\mathbb{N} = \{ 1, 2, \dotsc \}$, $\mathbb{N}_0 = \{ 0, 1, 2, \dotsc \}$, and $\#$ is the cardinality of a finite set. Following the usual convention, we define $\pm \infty+c \coloneqq \pm \infty$ and ${c}/{\infty} \coloneqq 0$ for any constant $c \in \mathbb{R}$, $c \cdot \infty \coloneqq \infty$, $c^\infty \coloneqq \infty$, and $c^{-\infty} \coloneqq 0$ if $c>0$, and $c \cdot \infty \coloneqq - \infty$ if $c<0$. The logarithm without base, denoted $\log$, will always mean the natural logarithm.

\section{Preliminaries} \label{preliminaries}

In this section, we introduce some elementary facts about Pierce expansions and some useful tools for calculating the Hausdorff dimension, which will be used in proving the main results.

\begin{proposition} [{\cite[Proposition 2.1]{Ahn23}}] \label{digit condition lemma}
Let $x \in \mathbb{I}$. For each $n \in \mathbb{N}$, the following hold:
\begin{enumerate}[label=\upshape(\roman*), ref=\roman*, leftmargin=*, widest=ii]
\item \label{digit condition lemma 1}
$d_{n+1}(x) > d_n(x)$.
\item \label{digit condition lemma 2}
$d_n(x) \geq n$.
\end{enumerate}
\end{proposition}

\begin{proof}
Both parts follow from \eqref{Pierce algorithm 1} and \eqref{Pierce algorithm 2}. See \cite{Ahn23} for the details.
\end{proof}

We shall introduce a symbolic space which is closely related to the Pierce expansion digit sequences. Let $\Sigma_0 \coloneqq \{ \epsilon \}$, where $\epsilon$ denotes the empty word. Define
\begin{align*}
\Sigma_n &\coloneqq \{ (\sigma_k)_{k=1}^n \in \mathbb{N}^n : \sigma_k < \sigma_{k+1} \text{ for all } 1 \leq k \leq n-1 \} \quad \text{for } n \in \mathbb{N}, \\
\Sigma_\infty &\coloneqq \{ (\sigma_k)_{k \in \mathbb{N}} \in \mathbb{N}^{\mathbb{N}} : \sigma_k < \sigma_{k+1} \text{ for all } k \in \mathbb{N} \}.
\end{align*}
Finally, let
\[
\Sigma \coloneqq \Sigma_0 \cup \bigcup_{n \in \mathbb{N}} \Sigma_n \cup \Sigma_\infty.
\]

We say that a finite sequence $(\sigma_k)_{k=1}^n \in \mathbb{N}^n$ for some $n \in \mathbb{N}$ is {\em Pierce admissible} if there exists $x \in (0,1]$ such that $d_k(x) = \sigma_k$ for all $1 \leq k \leq n$. Here, the digit sequence of $x$ need not be of length $n$. For an infinite sequence $(\sigma_k)_{k \in \mathbb{N}} \in \mathbb{N}^{\mathbb{N}}$, it is said to be {\em Pierce admissible} if $(\sigma_k)_{k=1}^n$ is Pierce admissible for all $n \in \mathbb{N}$. In other words, $(\sigma_k)_{k \in \mathbb{N}} \in \mathbb{N}^{\mathbb{N}}$ is Pierce admissible if and only if there exists $x \in (0,1]$ such that $d_k(x) = \sigma_k$ for all $k \in \mathbb{N}$. We denote by $\Sigma_{\admissible}$ the collection of all Pierce admissible sequences.

\begin{proposition} [{\cite[Proposition 2.2]{Fan15}}] \label{strict increasing condition}
We have $\Sigma_{\admissible} = \Sigma \setminus \Sigma_0$.
\end{proposition}

\begin{proposition} [{See \cite[p.~392]{Ahn23}}] \label{strict increasing lemma}
For each $\sigma \in \Sigma_{\admissible}$, one of the following holds:
\begin{enumerate}[label=\upshape(\roman*), ref=\roman*, leftmargin=*, widest=ii]
\item \label{strict increasing lemma 1}
$\sigma \coloneqq (\sigma_k)_{k=1}^n \in \Sigma_n$ for some $n \in \mathbb{N}$, in which case $\sigma_k \geq k$ for all $1 \leq k \leq n$.
\item \label{strict increasing lemma 2}
$\sigma \coloneqq (\sigma_k)_{k \in \mathbb{N}} \in \Sigma_\infty$, in which case $\sigma_k \geq k$ for all $k \in \mathbb{N}$.
\end{enumerate}
\end{proposition}

\begin{proof}
By Proposition \ref{strict increasing condition}, for any $\sigma \in \Sigma_{\admissible}$, either $\sigma \in \Sigma_n$ for some $n \in \mathbb{N}$ or $\sigma \in \Sigma_\infty$. Part (\ref{strict increasing lemma 1}) follows from the definition of $\Sigma_n$, and (\ref{strict increasing lemma 2}) from the definition of $\Sigma_\infty$. 
\end{proof}

For the empty word $\epsilon \in \Sigma_0$, we define $\epsilon^{(j)} \coloneqq (j) \in \Sigma_1$ for each $j \in \mathbb{N}$. For each $n \in \mathbb{N}$ and $\sigma \coloneqq (\sigma_k)_{k=1}^n \in \Sigma_n$, we define some finite sequences associated with $\sigma$. Let $\underline{\sigma} \coloneqq (\tau_k)_{k=1}^n \in \Sigma_n$ and $\sigma^{(j)} \coloneqq (\upsilon_k)_{k=1}^{n+1} \in \Sigma_{n+1}$ for each positive integer $j >\sigma_n$ be given by
\[
\tau_k \coloneqq \begin{cases} \sigma_k, &\text{if } 1 \leq k \leq n-1; \\ \sigma_n+1, &\text{if } k = n, \end{cases}
\quad \text{and} \quad
\upsilon_k \coloneqq \begin{cases} \sigma_k, &\text{if } 1 \leq k \leq n; \\ j, &\text{if } k = n+1, \end{cases}
\]
respectively, i.e.,
\begin{align} \label{definition of underline sigma}
\underline{\sigma} = (\sigma_1, \dotsc, \sigma_{n-1}, \sigma_n + 1),
\quad \text{and} \quad
\sigma^{(j)} = (\sigma_1, \dotsc, \sigma_{n-1}, \sigma_n, j).
\end{align}

Let $\varphi_0 (\epsilon) \coloneqq 0$. For each $n \in \mathbb{N}$ and $\sigma \coloneqq (\sigma_k)_{k=1}^n \in \mathbb{N}^n$, we define
\begin{align} \label{definition of varphi n}
\varphi_n (\sigma) 
\coloneqq \sum_{j=1}^n \left( (-1)^{j+1} \prod_{k=1}^j \frac{1}{\sigma_k} \right).
\end{align}
It is clear that the map $\varphi_n \colon \mathbb{N}^n \to \mathbb{R}$ is well-defined as a sum of fintely many terms. In fact, $\varphi_n$ is modelled on the $n$th partial sum of the Pierce expansion \eqref{Pierce expansion}.

\begin{proposition}[See {\cite[Proposition 2.3]{Ahn23}} and {\cite[Proposition 2.1]{Fan15}}]
Let $n \in \mathbb{N}$. For each $\sigma \coloneqq (\sigma_k)_{k=1}^n \in \Sigma_n$, 
\begin{align} \label{two sequences for rational number}
\varphi_n (\underline{\sigma}) = \varphi_{n+1} (\sigma^{(\sigma_n+1)}).
\end{align}
\end{proposition}

\begin{proof}
Let $\sigma \coloneqq (\sigma_k)_{k=1}^n \in \Sigma_n$. Put $\tau \coloneqq (\sigma_k)_{k=1}^{n-1} \in \Sigma_{n-1}$. Then, by \eqref{definition of underline sigma} and \eqref{definition of varphi n}, we find that 
\begin{align*}
\varphi_n (\underline{\sigma})
&= \varphi_{n-1} (\tau) + (-1)^{n+1} \left( \prod_{k=1}^{n-1} \frac{1}{\sigma_k} \right) \frac{1}{\sigma_n+1} \\
&= \varphi_{n-1} (\tau) + (-1)^{n+1} \left( \prod_{k=1}^{n-1} \frac{1}{\sigma_k} \right) \left( \frac{1}{\sigma_n} - \frac{1}{\sigma_n (\sigma_n+1)} \right) \\
&= \varphi_{n-1} (\tau) + (-1)^{n+1} \left( \prod_{k=1}^{n} \frac{1}{\sigma_k} \right) + (-1)^{n+2} \left( \prod_{k=1}^{n} \frac{1}{\sigma_k} \right) \frac{1}{\sigma_n+1} 
= \varphi_{n+1} (\sigma^{(\sigma_n+1)}),
\end{align*}
as desired.
\end{proof}

For each $n \in \mathbb{N}$ and $\sigma \coloneqq (\sigma_k)_{k=1}^n \in \Sigma_n$, we define the {\em $n$th level fundamental interval} by
\[
\mathcal{I}(\sigma) \coloneqq \{ x \in (0,1] : d_k(x) = \sigma_k \text{ for all } 1 \leq k \leq n \}.
\]
The following proposition justifies the naming of $\mathcal{I}(\sigma)$ as an interval.

\begin{proposition}[See {\cite[Proposition 3.5]{Ahn23}} and {\cite[Theorem 1]{Sha86}}] \label{I sigma}
Let $n \in \mathbb{N}$ and $\sigma \coloneqq (\sigma_k)_{k=1}^n \in \Sigma_n$. Then $\mathcal{I} (\sigma)$ is an interval with endpoints $\varphi_n (\sigma)$ and $\varphi_n (\underline{\sigma})$. More precisely,
\begin{align} \label{I sigma endpoints}
\mathcal{I}(\sigma)  = \begin{cases}
(\varphi_n ({\underline{\sigma}}), \varphi_n (\sigma)], &\text{if $n$ is odd;} \\
[\varphi_n (\sigma), \varphi_n ({\underline{\sigma}})), &\text{if $n$ is even,}
\end{cases}
\quad \text{or} \quad
\mathcal{I}(\sigma)  = \begin{cases}
(\varphi_n ({\underline{\sigma}}), \varphi_n (\sigma)), &\text{if $n > 1$ is odd;} \\
(\varphi_n (\sigma), \varphi_n ({\underline{\sigma}})), &\text{if $n$ is even,}
\end{cases}
\end{align}
according as $\sigma_n > \sigma_{n-1} + 1$ or $\sigma_n = \sigma_{n-1}+1$ if $n>1$. Consequently,
\begin{align} \label{diam I sigma}
\len \mathcal{I} (\sigma) = | \varphi_n (\sigma) - \varphi_n (\underline{\sigma})| = \left( \prod_{k=1}^{n} \frac{1}{\sigma_k} \right) \frac{1}{\sigma_n+1} .
\end{align}
\end{proposition}

We will often use the following proposition to obtain an upper bound of the Hausdorff dimension.

\begin{proposition}[{\cite[Proposition 4.1]{Fal14}}] \label{box dimension}
Let $\mathcal{E}$ be a non-empty subset of $\mathbb{R}$. Let $(\mathcal{F}_n)_{n \in \mathbb{N}}$ be a sequence of finite coverings of $\mathcal{E}$. If for each $n \in \mathbb{N}$, we have $\diam F \leq \delta_n$ for every $F \in \mathcal{F}_n$, with $\delta_n \to 0$ as $n \to \infty$, then
\[
\hdim \mathcal{E} \leq \liminf_{n \to \infty} \frac{\log (\# \mathcal{F}_n)}{\log (1/\delta_n)}.
\]
\end{proposition}

The following proposition will play a key role in finding a lower bound of the Hausdorff dimension.

\begin{proposition}[{\cite[Example 4.6(a)]{Fal14}}] \label{gap dimension}
Let $\mathcal{E}_0$ be a closed subinterval of $[0,1]$, $( \mathcal{E}_k )_{k \in \mathbb{N}_0}$ a decreasing sequence of sets, and $\mathcal{E} \coloneqq \bigcap_{k \in \mathbb{N}_0} \mathcal{E}_k$. Suppose that each $\mathcal{E}_k$, $k \in \mathbb{N}$, is a union of a finite number of disjoint closed intervals, called $k$th level basic intervals, and each $(k-1)$th level basic interval in $\mathcal{E}_{k-1}$ contains $m_k$ intervals of $\mathcal{E}_{k}$ which are separated by gaps of lengths at least $\varepsilon_k$. If $m_k \geq 2$ and $\varepsilon_{k} > \varepsilon_{k+1} > 0$ for each $k \in \mathbb{N}$, then
\[
\hdim \mathcal{E} \geq \liminf_{n \to \infty} \frac{\log \left( \prod \limits_{k=1}^n m_k \right)}{\log (1/(m_{n+1} \varepsilon_{n+1}))}.
\]
\end{proposition}

We will use the following well-known bounds for the factorial in the calculation of limits. These bounds are coarser than the famous Stirling's formula, but the proof is elementary, and they are satisfactory for our arguments. First note that since the map $x \mapsto \log x$ is increasing on $(0, \infty)$, for any $n, m \in \mathbb{N}$ with $n \geq m$ we have
\begin{align} \label{bounds for sum of log}
\log m + \int_m^n \log x \, dx \leq \sum_{k=m}^n \log k \leq \log m + \int_{m+1}^{n+1} \log x \, dx.
\end{align}

\begin{proposition} [{\cite[Lemma 10.1]{KL16}}] \label{bounds for factorial lemma}
For each $n \in \mathbb{N}$, we have
\begin{align} \label{bounds for factorial}
\frac{n^n}{e^{n-1}} \leq n! \leq \frac{n^{n+1}}{e^{n-1}}.
\end{align}
\end{proposition}

\begin{proof}
The proof makes use of \eqref{bounds for sum of log}. See \cite{KL16} for the details.
\end{proof}

For later use, we record one relation between the difference of integer parts and the integer part of the difference of two real numbers. For any $x, y \in \mathbb{R}$, we have
\begin{align} \label{floor function subtraction}
\lfloor x \rfloor - \lfloor y \rfloor = \begin{cases} \lfloor x-y \rfloor, &\text{if } \{ x \} \geq \{ y \}; \\ \lfloor x-y \rfloor + 1, &\text{if } \{ x \} < \{ y \}, \end{cases}
\end{align}
where $\{ z \}$ denotes the fractional part of $z \in \mathbb{R}$, i.e., $\{ z \} \coloneqq z - \lfloor z \rfloor$.

\section{Auxiliary results} \label{auxiliary results}

In this section we derive some auxiliary results on the Hausdorff dimensions of certain subsets of $(0,1]$.

\begin{lemma} \label{rough dimension of E lemma}
Let $(l_n)_{n \in \mathbb{N}}$ and $(r_n)_{n \in \mathbb{N}}$ be sequences of positive real numbers satisfying
\begin{enumerate}[label=\upshape(\roman*), ref=\roman*, leftmargin=*, widest=iii]
\item \label{rough dimension of E lemma 1}
$\Delta_n \coloneqq r_n-l_n \geq 2$,
\item \label{rough dimension of E lemma 2}
$r_{n+1} \leq 2 l_{n+1}-1$,
\item \label{rough dimension of E lemma 3}
$r_n \leq l_{n+1}$,
\end{enumerate}
for each $n \in \mathbb{N}$. Define a subset $\mathcal{E} = \mathcal{E} ((l_n)_{n \in \mathbb{N}}, (r_n)_{n \in \mathbb{N}})$ of $(0,1]$ by
\begin{align} \label{definition of E ln rn}
\mathcal{E} \coloneqq \{ x \in (0,1] : l_n < d_n(x) \leq r_n \text{ for all } n \in \mathbb{N} \}.
\end{align}
Then
\begin{align} \label{rough dimension of E ln rn}
\liminf_{n \to \infty} \frac{\log \left( \prod \limits_{k=1}^n \Delta_k \right)}{\log \left( \dfrac{r_{n+1} r_{n+2}}{\Delta_{n+1} \Delta_{n+2}} \prod \limits_{k=1}^{n+1} r_k \right)} 
\leq \hdim \mathcal{E}
\leq \liminf_{n \to \infty} \frac{\log \left( \prod \limits_{k=1}^n \Delta_k \right)}{\log \left( \dfrac{r_{n+1}}{\Delta_{n+1}} \prod \limits_{k=1}^{n+1} l_k \right)}.
\end{align}
\end{lemma}

\begin{proof}
For each $n \in \mathbb{N}$, let
\[
\Upsilon_n \coloneqq \{ (\sigma_k)_{k=1}^n \in \Sigma_n : l_k < \sigma_k \leq r_k \text{ for all } 1 \leq k \leq n \}.
\]
For each $1 \leq k \leq n$, since the bounds for $\sigma_k$ in $\Upsilon_n$ differ by $\Delta_k$, which is strictly greater than $1$ by condition (\ref{rough dimension of E lemma 1}), there is at least one positive integer in the interval $(l_k, r_k]$. Hence, $\Upsilon_n \neq \varnothing$.

Define $\Upsilon_0 \coloneqq \Sigma_0$ and
\[
\Upsilon \coloneqq \bigcup_{n \in \mathbb{N}_0} \Upsilon_n.
\]
For each $\sigma \in \Upsilon_n$, $n \in \mathbb{N}_0$, we define the {\em $n$th level basic interval} $\mathcal{J}(\sigma)$ by
\begin{align} \label{J sigma}
\begin{aligned}
\mathcal{J} (\sigma) 
&\coloneqq \bigcup_{j = \lfloor l_{n+1} \rfloor + 1}^{\lfloor r_{n+1} \rfloor} \overline{\mathcal{I} (\sigma^{(j)})} \\
&= \begin{cases}
\big[ \varphi_{n+1} (\sigma^{(\lfloor l_{n+1} \rfloor + 1)}), \varphi_{n+1} (\sigma^{(\lfloor r_{n+1} \rfloor + 1)}) \big], &\text{if $n$ is odd}; \\
\big[ \varphi_{n+1} (\sigma^{(\lfloor r_{n+1} \rfloor + 1)}), \varphi_{n+1} (\sigma^{(\lfloor l_{n+1} \rfloor + 1)}) \big], &\text{if $n$ is even},
\end{cases} 
\end{aligned}
\end{align}
where the second equality follows from \eqref{I sigma endpoints} and \eqref{definition of underline sigma}.
Notice that 
\begin{align*}
x \in \mathcal{E}
&\iff \forall n \in \mathbb{N}, \, l_n < d_n(x) \leq r_n \\
&\iff \forall n \in \mathbb{N}, \, \lfloor l_n \rfloor + 1 \leq d_n(x) \leq \lfloor r_n \rfloor \\
&\iff \forall n \in \mathbb{N}_0, \, \exists \sigma \in \Upsilon_n \text{ such that } x \in {\textstyle \bigcup_{j = \lfloor l_{n+1} \rfloor + 1}^{\lfloor r_{n+1} \rfloor}} \mathcal{I} (\sigma^{(j)}) \\
&\iff \forall n \in \mathbb{N}_0, \, \exists \sigma \in \Upsilon_n \text{ such that } x \in \mathcal{J} (\sigma),
\end{align*}
where the second equivalence is due to the fact that the $d_n(x)$ are integers, the third to Proposition \ref{I sigma}, and the fourth to \eqref{J sigma} and to the fact that the last two statements hold only when $x \in \mathbb{I}$ so that neither including nor excluding endpoints makes any difference. Thus,
\begin{align} \label{equivalent definition of E ln rn}
\mathcal{E} = \bigcap_{n \in \mathbb{N}_0} \mathcal{E}_n, \quad \text{where} \quad \mathcal{E}_n \coloneqq \bigcup_{\sigma \in \Upsilon_n} \mathcal{J} (\sigma).
\end{align}

For each $\sigma \coloneqq (\sigma_k)_{k=1}^n \in \Upsilon_n$, $n \in \mathbb{N}_0$, by \eqref{J sigma} and \eqref{definition of varphi n}, we find that
\begin{align*}
\diam \mathcal{J} (\sigma) = \len \mathcal{J} (\sigma) 
= \left( \prod_{k=1}^n \frac{1}{\sigma_k} \right) \left( \frac{1}{\lfloor l_{n+1} \rfloor + 1} - \frac{1}{\lfloor r_{n+1} \rfloor + 1} \right).
\end{align*}
Since $l_k < \sigma_k \leq r_k$ for $1 \leq k \leq n$ and $x-1 < \lfloor x \rfloor \leq x$ for any $x \in \mathbb{R}$, we obtain 
\begin{align} \label{diam J sigma}
\begin{aligned}
\diam \mathcal{J} (\sigma)
&< \left( \prod_{k=1}^n \frac{1}{l_k} \right) \left( \frac{1}{l_{n+1}} - \frac{1}{r_{n+1}+1} \right)  
= \left( \prod_{k=1}^{n+1} \frac{1}{l_k} \right) \frac{\Delta_{n+1}+1}{r_{n+1}+1} \\
&<2 \left( \prod_{k=1}^{n+1} \frac{1}{l_k} \right) \frac{\Delta_{n+1}}{r_{n+1}},
\end{aligned}
\end{align}
where the second inequality follows from (\ref{rough dimension of E lemma 1}).

For each $k \in \mathbb{N}$, let $m_k$ be the number of $k$th level basic intervals contained in a single $(k-1)$th level basic interval. For each $\sigma \in \Upsilon_{k-1}$, by \eqref{J sigma}, $\mathcal{J}(\sigma)$ is the union of $\lfloor r_k \rfloor - \lfloor l_k \rfloor$ closures of $k$th level fundamental intervals, each of which contains exactly one $k$th level basic interval. So, $m_k = \lfloor r_k \rfloor - \lfloor l_k \rfloor$ for each $k \in \mathbb{N}$, and 
\begin{align} \label{cardinality of Dn}
\# \Upsilon_n = \prod_{k=1}^n m_k
\end{align}
for each $n \in \mathbb{N}_0$. Since $m_k$ is equal to either $\lfloor \Delta_k \rfloor$ or $\lfloor \Delta_k \rfloor + 1$ by \eqref{floor function subtraction}, we have 
\[
\prod_{k=1}^n (\Delta_k-1) < \prod_{k=1}^n \lfloor \Delta_k \rfloor \leq \# \Upsilon_n \leq \prod_{k=1}^n (\lfloor \Delta_k \rfloor + 1)  \leq \prod_{k=1}^n (\Delta_k+1).
\]
Put $\Delta \coloneqq \inf_{k \in \mathbb{N}} \Delta_k$, which is finite by (\ref{rough dimension of E lemma 1}). Then $\Delta_k \geq \Delta > 1$ for all $k \in \mathbb{N}$ by the same condition. Note that the map $x \mapsto x/(x-1)$ defined on $(1, \infty)$ assumes every value strictly greater than $1$. So, we can find a constant $c = c(\Delta) >1$ such that $\Delta > c/(c-1) > 1$. Then $\Delta_k > c/(c-1)$, or $c^{-1} \Delta_k < \Delta_k-1$ for all $k \in \mathbb{N}$. Since $c>1$, we also have $\Delta_k > 1/(c-1)$, so $\Delta_k+1 < c \Delta_k$, for all $k \in \mathbb{N}$. Consequently,
\begin{align} \label{bounds for Dn}
\frac{1}{c^n} \prod_{k=1}^n \Delta_k \leq \# \Upsilon_n \leq c^n \prod_{k=1}^n  \Delta_k
\end{align}
for any $n \in \mathbb{N}_0$.

We establish upper and lower bounds for $\hdim \mathcal{E}$ separately.

{\sc Upper bound}.
We make use of Proposition \ref{box dimension} to find the desired upper bound. By \eqref{equivalent definition of E ln rn}, for each $n \in \mathbb{N}_0$, the collection of all $n$th level basic intervals $\{ \mathcal{J} (\sigma) : \sigma \in \Upsilon_n \}$ is a covering of $\mathcal{E}$. It is clear from \eqref{diam J sigma} that $\delta_n \coloneqq \max \{ \diam \mathcal{J} (\sigma) : \sigma \in \Upsilon_n \} \to 0$ as $n \to \infty$. Hence, by \eqref{diam J sigma} and \eqref{bounds for Dn}, we obtain
\begin{align*}
\hdim \mathcal{E}
&\leq \liminf_{n \to \infty} \frac{\log (\# \{ \mathcal{J} (\sigma) : \sigma \in \Upsilon_n \})}{\log \left( 1 / \delta_n \right)} 
= \liminf_{n \to \infty} \frac{\log (\# \Upsilon_n)}{\log \left( 1 / \delta_n \right)}
\\
&\leq \liminf_{n \to \infty} \frac{n \log c + \log \left( \prod \limits_{k=1}^n \Delta_k \right)}{-\log 2 + \log \left( \dfrac{r_{n+1}}{\Delta_{n+1}} \prod \limits_{k=1}^{n+1} l_k \right)} 
= \liminf_{n \to \infty} \frac{\log \left( \prod \limits_{k=1}^n \Delta_k \right)}{\log \left( \dfrac{r_{n+1}}{\Delta_{n+1}} \prod \limits_{k=1}^{n+1} l_k \right)}.
\end{align*}
To see the last equality, first note that $r_{n+1} > \Delta_{n+1}$ by definition and that $l_k \geq k$ for all $k \geq 2$ by (\ref{rough dimension of E lemma 1}) and (\ref{rough dimension of E lemma 3}). Then, for all sufficiently large $n$,
\begin{align*}
0 \leq \frac{n}{\log \left( \dfrac{r_{n+1}}{\Delta_{n+1}} \prod \limits_{k=1}^{n+1} l_k \right)}
\leq \frac{n}{\log \left( \prod \limits_{k=1}^{n+1} l_k \right)}
\leq \frac{n}{\log l_1 + \log ( n+1 ) !},
\end{align*}
with the rightmost term tending to $0$ as $n \to \infty$ by \eqref{bounds for factorial}.

{\sc Lower bound}.
We will use Proposition \ref{gap dimension} to obtain the result. Recall from \eqref{equivalent definition of E ln rn} that $\mathcal{E}_n = \bigcup_{\sigma \in \Upsilon_n} \mathcal{J} (\sigma)$ for each $n \in \mathbb{N}_0$. By \eqref{J sigma} and \eqref{definition of underline sigma}, $\mathcal{E}_0$ equals $\mathcal{J}(\epsilon) = [ (\lfloor r_{1}\rfloor +1)^{-1}, (\lfloor l_{1}\rfloor+1)^{-1} ]$, which is a non-degenerate closed subinterval of $[0,1]$ since $(\lfloor r_{1} \rfloor + 1 ) - (\lfloor l_{1} \rfloor + 1) \geq \lfloor \Delta_1 \rfloor \geq 2$ by \eqref{floor function subtraction} and (\ref{rough dimension of E lemma 1}).  

Fix $n \in \mathbb{N}$, and consider $\mathcal{E}_n$. By \eqref{J sigma} and \eqref{bounds for Dn}, we observe that $\mathcal{E}_n$ is a disjoint union of finitely many closed intervals. Note that $(\lfloor r_n \rfloor-1)  - (\lfloor l_n \rfloor + 1) \geq \lfloor \Delta_n \rfloor -2 \geq 0$ by \eqref{floor function subtraction} and (\ref{rough dimension of E lemma 1}). So, we can take a $\sigma \coloneqq (\sigma_k)_{k=1}^n \in \Upsilon_n$ satisfying $\lfloor l_n \rfloor + 1 \leq \sigma_n \leq \lfloor r_n \rfloor-1$. Then $\underline{\sigma} \in \Sigma_n$ is also in $\Upsilon_n$ since its $n$th term $\sigma_n+1$ satisfies $\lfloor l_n \rfloor + 2 \leq \sigma_n+1 \leq \lfloor r_n \rfloor$ (see \eqref{definition of underline sigma} for the definition of $\underline{\sigma}$). Hence, $\mathcal{J}(\sigma)$ and $\mathcal{J}(\underline{\sigma})$ are two adjacent $n$th level basic intervals contained in the same $(n-1)$th level basic interval, namely, $\mathcal{J}( (\sigma_k)_{k=1}^{n-1} )$. Conversely, it is clear that any two adjacent $n$th level basic intervals in the $(n-1)$th level basic interval $\mathcal{J}( (\sigma_k)_{k=1}^{n-1} )$ are of these forms.

We need to estimate the length of the gap between $\mathcal{J}(\sigma)$ and $\mathcal{J}(\underline{\sigma})$. Obviously, the gap is an interval. We denote by $\mathcal{G}(\sigma)$ its closure. Suppose first that $n \in \mathbb{N}$ is odd. Then $\varphi_n (\underline{\sigma})$ is the left endpoint of $\mathcal{I} (\sigma)$ as well as the right endpoint of $\mathcal{I} (\underline{\sigma})$ in view of \eqref{I sigma endpoints}, and this tells us that $\mathcal{I} (\sigma)$ is on the right of $\mathcal{I} (\underline{\sigma})$. Since $\mathcal{J} (\tau) \subseteq \mathcal{I} (\tau)$ for any $\tau \in \Upsilon_n$ by \eqref{J sigma}, we deduce that $\mathcal{J} (\sigma)$ is on the right of $\mathcal{J} (\underline{\sigma})$. Then $\mathcal{G} (\sigma)$ is the closed interval between the left endpoint of $\mathcal{J} (\sigma)$ and the right endpoint of $\mathcal{J} (\underline{\sigma})$; precisely, by \eqref{J sigma},
\begin{align*} 
\mathcal{G} (\sigma) 
&= \big[ \varphi_{n+1} ({\underline{\sigma}^{(\lfloor r_{n+1} \rfloor+1)}}), \varphi_{n+1} (\sigma^{(\lfloor l_{n+1} \rfloor+1)}) \big].
\end{align*}
If $n \in \mathbb{N}$ is even, then by a similar argument we infer that $\mathcal{J} (\sigma)$ is on the left of $\mathcal{J} (\underline{\sigma})$. Hence, $\mathcal{G} (\sigma)$ is the closed interval between the right endpoint of $\mathcal{J} (\sigma)$ and the left endpoint of $\mathcal{J} (\underline{\sigma})$; precisely, by \eqref{J sigma},
\begin{align*} 
\mathcal{G} (\sigma)
&= \big[ \varphi_{n+1} (\sigma^{(\lfloor l_{n+1} \rfloor+1)}), \varphi_{n+1} ({\underline{\sigma}^{(\lfloor r_{n+1} \rfloor+1)}}) \big].
\end{align*}

By the definition of $\mathcal{G}(\sigma)$, we know that $\varphi_n (\underline{\sigma}) = \varphi_{n+1} (\sigma^{(\sigma_n+1)}) \in \mathcal{G}(\sigma)$, where the equality follows from \eqref{two sequences for rational number}. Then, regardless of whether $n$ is odd or even, for any $\sigma \coloneqq (\sigma_k)_{k=1}^n \in \Upsilon_n$ with $\lfloor l_n \rfloor + 1 \leq \sigma_n \leq \lfloor r_n \rfloor - 1$, we have
\begin{align*}
\len \mathcal{G} (\sigma) = \big| \varphi_{n+1} ({\underline{\sigma}^{(\lfloor r_{n+1} \rfloor+1)}}) - \varphi_n (\underline{\sigma}) \big| + \big| \varphi_{n+1} (\sigma^{(\sigma_n+1)}) - \varphi_{n+1} (\sigma^{(\lfloor l_{n+1} \rfloor+1)}) \big|.
\end{align*}
Notice that by \eqref{definition of underline sigma} the initial $n$ terms of ${\underline{\sigma}^{(\lfloor r_{n+1} \rfloor+1)}}$ coincide with those of $\underline{\sigma}$ and that $\sigma^{(\sigma_n+1)}$ and $\sigma^{(\lfloor l_{n+1} \rfloor+1)}$ share the initial $n$ terms, which agree with those of $\sigma$. It follows by \eqref{definition of varphi n} that
\begin{align*}
\len \mathcal{G} (\sigma)
&= \left( \prod_{k=1}^{n-1} \frac{1}{\sigma_k} \right) \frac{1}{\sigma_n+1} \frac{1}{\lfloor r_{n+1} \rfloor + 1}   +  \left( \prod_{k=1}^{n} \frac{1}{\sigma_k} \right) \left( \frac{1}{\sigma_n+1} - \frac{1}{\lfloor l_{n+1} \rfloor + 1} \right) \\
&\geq \left( \prod_{k=1}^{n} \frac{1}{\lfloor r_k \rfloor} \right) \frac{1}{\lfloor r_{n+1} \rfloor + 1} 
 + \left( \prod_{k=1}^{n-1} \frac{1}{\lfloor r_k \rfloor} \right) \frac{1}{\lfloor r_n \rfloor-1} \left( \frac{1}{\lfloor r_n \rfloor} - \frac{1}{\lfloor l_{n+1} \rfloor + 1}  \right) \\
&= \min \{ \len \mathcal{G}(\sigma) : \sigma \coloneqq (\sigma_k)_{k=1}^n \in {\Upsilon}_n \text{ with } \lfloor l_n \rfloor + 1 \leq \sigma_n \leq \lfloor r_n \rfloor-1 \},
\end{align*}
since $\lfloor l_k \rfloor + 1 \leq \sigma_k \leq \lfloor r_k \rfloor$ for $1 \leq k \leq n-1$ and $\lfloor l_n \rfloor + 1 \leq \sigma_n \leq \lfloor r_n \rfloor-1$, and hence the minimum is attained by taking
\[
\sigma \coloneqq ( \lfloor r_1 \rfloor, \lfloor r_2 \rfloor, \dotsc, \lfloor r_{n-1} \rfloor, \lfloor r_n \rfloor -1 ).
\]
Now, since $x-1 < \lfloor x \rfloor \leq x$ for any $x \in \mathbb{R}$, the length of the gap between $\mathcal{J} (\sigma)$ and $\mathcal{J} (\underline{\sigma})$ is at least
\begin{align*}
&\min \{ \len \mathcal{G}(\sigma) : \sigma 
\coloneqq (\sigma_k)_{k=1}^n \in {\Upsilon}_n \text{ with } \lfloor l_n \rfloor + 1 \leq \sigma_n \leq \lfloor r_n \rfloor-1 \} \\
&\hspace{0.5cm} > \left( \prod_{k=1}^{n} \frac{1}{r_k} \right) \frac{1}{r_{n+1} + 1} 
 + \left( \prod_{k=1}^{n-1} \frac{1}{r_k} \right) \frac{1}{r_n-1} \left( \frac{1}{r_n} - \frac{1}{l_{n+1}}  \right) \\
&\hspace{0.5cm} > \left( \prod_{k=1}^{n} \frac{1}{r_k} \right) \frac{1}{r_{n+1} + 1} 
 + \left( \prod_{k=1}^{n} \frac{1}{r_k} \right) \left( \frac{1}{r_n} - \frac{1}{l_{n+1}}  \right) 
= \left( \prod_{k=1}^{n} \frac{1}{r_k} \right) \left( \frac{1}{r_n} - \frac{r_{n+1}-l_{n+1}+1}{l_{n+1}(r_{n+1}+1)} \right) 
\\
&\hspace{0.5cm} \geq \left( \prod_{k=1}^{n} \frac{1}{r_k} \right) \left( \frac{1}{r_n} - \frac{1}{r_{n+1}+1} \right) 
= \left( \prod_{k=1}^{n} \frac{1}{r_k} \right) \frac{(r_{n+1}-l_{n+1})+(l_{n+1}-r_n)+1}{r_n(r_{n+1}+1)}  \\
&\hspace{0.5cm} > \left( \prod_{k=1}^{n} \frac{1}{r_k} \right) \frac{\Delta_{n+1}}{r_n (r_{n+1}+1)} 
\geq \frac{1}{2} \left( \prod_{k=1}^{n} \frac{1}{r_k} \right) \frac{\Delta_{n+1}}{r_n r_{n+1}} 
\eqqcolon \varepsilon_n ,
\end{align*}
where the third inequality follows from (\ref{rough dimension of E lemma 2}), the second-to-last from (\ref{rough dimension of E lemma 3}), and the last one from the fact that $r_{n+1} \geq 1$ for all $n \in \mathbb{N}$.

Observe that $(\varepsilon_n)_{n \in \mathbb{N}}$ is a decreasing sequence of positive numbers since
\[
\frac{\varepsilon_n}{\varepsilon_{n+1}} = \frac{r_{n+1}}{r_n} \frac{r_{n+2}}{\Delta_{n+2}} \Delta_{n+1} > \frac{l_{n+1}+2}{r_n+2} \frac{r_{n+2}}{r_{n+2}-l_{n+2}} \cdot 2 > 1 \cdot 1 \cdot 2 > 1
\]
for each $n \in \mathbb{N}$ by (\ref{rough dimension of E lemma 1}) and (\ref{rough dimension of E lemma 3}). Recall that $m_n = \lfloor r_n \rfloor - \lfloor l_n \rfloor$ for any $n \in \mathbb{N}$. Hence, for each $n \in \mathbb{N}$, we have $m_n \geq \lfloor \Delta_n \rfloor \geq 2$ by \eqref{floor function subtraction} and (\ref{rough dimension of E lemma 1}). Furthermore, since $\Delta_{n+1} \geq 2$ by (\ref{rough dimension of E lemma 1}), we have $m_{n+1} \geq \lfloor \Delta_{n+1} \rfloor > \Delta_{n+1}/2$ for all $n \in \mathbb{N}$. Therefore, using Proposition \ref{gap dimension}, \eqref{cardinality of Dn}, and \eqref{bounds for Dn}, we find that
\begin{align*}
\hdim {\mathcal{E}}
&\geq \liminf_{n \to \infty} \frac{\log \left( \prod \limits_{k=1}^n m_k \right)}{\log (1/(m_{n+1} \varepsilon_{n+1}))} 
= \liminf_{n \to \infty} \frac{\log \left( \# \Upsilon_n \right)}{\log (1/(m_{n+1} \varepsilon_{n+1}))} \\
&\geq \liminf_{n \to \infty} \frac{\log \left( \dfrac{1}{c^n} \prod \limits_{k=1}^n \Delta_k \right)}{\log \left( \dfrac{4 r_{n+1} r_{n+2}}{\Delta_{n+1} \Delta_{n+2}} \prod \limits_{k=1}^{n+1} r_k  \right)}
= \liminf_{n \to \infty} \frac{\log \left( \prod \limits_{k=1}^n \Delta_k \right)}{\log \left( \dfrac{r_{n+1} r_{n+2}}{\Delta_{n+1} \Delta_{n+2}} \prod \limits_{k=1}^{n+1} r_k \right)}.
\end{align*}
To see the last equality, first note that $r_k > \Delta_k$ by definition and that $r_k > k$ for all $k \in \mathbb{N}$ by (\ref{rough dimension of E lemma 1}) and (\ref{rough dimension of E lemma 3}). Then for all $n \in \mathbb{N}$,
\begin{align*}
0 \leq \frac{n}{\log \left( \dfrac{4 r_{n+1} r_{n+2}}{\Delta_{n+1} \Delta_{n+2}} \prod \limits_{k=1}^{n+1} r_k \right)}
\leq \frac{n}{\log ( n+1 ) !},
\end{align*}
with the rightmost term tending to $0$ as $n \to \infty$ by \eqref{bounds for factorial}.

This completes the proof of the lemma.
\end{proof}

We can make the lower and upper bounds for $\hdim \mathcal{E}$ in \eqref{rough dimension of E ln rn} agree for some $\mathcal{E} = \mathcal{E} ((l_n)_{n \in \mathbb{N}}, (r_n)_{n \in \mathbb{N}} )$. The following lemma gives one example. We remark that an analogous set for Engel expansions was first introduced by Liu and Jun Wu \cite[Section 2]{LW01} and further generalized by Shang and Min Wu \cite[Lemma 2.6]{SW21}. 

\begin{lemma} \label{dimension of  E star}
Let $(u_n)_{n \in \mathbb{N}}$ be a sequence of real numbers such that $2 \leq u_n \leq u_{n+1}$ for all $n \in \mathbb{N}$. Define a subset $\mathcal{E}^* =  \mathcal{E}^* ( (u_n)_{n \in \mathbb{N}} )$ of $(0,1]$ by
\begin{align} \label{definition of E un}
\mathcal{E}^* \coloneqq \{ x \in (0,1]: n u_n < d_n(x) \leq (n+1) u_n \text{ for all } n \in \mathbb{N} \}.
\end{align}
Then
\[
\hdim \mathcal{E}^* = (1+\eta)^{-1}, \quad \text{where} \quad \eta \coloneqq \limsup_{n \to \infty} \frac{n \log n + \log u_{n+1}}{\log \left( \prod \limits_{k=1}^n u_k \right)}.
\]
\end{lemma}

\begin{proof}
Define $l_n \coloneqq nu_n$ and $r_n \coloneqq (n+1) u_n$ for each $n \in \mathbb{N}$. Then $\mathcal{E}^* = \mathcal{E}$ and $\hdim \mathcal{E}^* = \hdim \mathcal{E}$, where $\mathcal{E} = \mathcal{E} ( (l_n)_{n \in \mathbb{N}}, (r_n)_{n \in \mathbb{N}} )$ is defined as in \eqref{definition of E ln rn}. To use \eqref{rough dimension of E ln rn}, we verify that the sequences $(l_n)_{n \in \mathbb{N}}$ and $(r_n)_{n \in \mathbb{N}}$ satisfy conditions (\ref{rough dimension of E lemma 1})--(\ref{rough dimension of E lemma 3}) in Lemma \ref{rough dimension of E lemma}. In fact, by assumptions on $(u_n)_{n \in \mathbb{N}}$,
\begin{enumerate}[label=\upshape(\roman*), leftmargin=*, widest=iii]
\item
$\Delta_n \coloneqq r_n - l_n = (n+1)u_n - n u_n = u_n \geq 2$,
\item
$2 l_{n+1} - r_{n+1} = 2(n+1)u_{n+1} - (n+2)u_{n+1} = n u_{n+1} \geq 2n \geq 1$,
\item
$l_{n+1} - r_n = (n+1)u_{n+1} - (n+1) u_n = (n+1)(u_{n+1}-u_n) \geq 0$
\end{enumerate}
for each $n \in \mathbb{N}$, as desired. 

Put $B \coloneqq \sup_{n \in \mathbb{N}} u_n$. Clearly, $B = \lim_{n \to \infty} u_n$ and $B \in [2,\infty]$, since $(u_n)_{n \in \mathbb{N}}$ is increasing and bounded below by $2$. Suppose $B \neq \infty$, and let $\varepsilon \in (0, B)$ be arbitrary. By definition of $B$, we can find an $M = M(\varepsilon) \in \mathbb{N}$ such that $B-\varepsilon < u_n \leq B$ for all $n > M$. Put $A \coloneqq \log \left( \prod_{k=1}^M u_k \right)$. Then, for all $n>M$,
\begin{align*}
\frac{n}{A + (n-M) \log B} 
\leq \frac{n}{\log \left( \prod \limits_{k=1}^n u_k \right)}
< \frac{n}{A + (n-M) \log (B-\varepsilon)},
\end{align*}
and so
\begin{align*}
\frac{1}{\log B} \leq \liminf_{n \to \infty} \frac{n}{\log \left( \prod \limits_{k=1}^n u_k \right)} \leq \limsup_{n \to \infty} \frac{n}{\log \left( \prod \limits_{k=1}^n u_k \right)} \leq \frac{1}{\log (B-\varepsilon)}.
\end{align*}
Since $\varepsilon \in (0, B)$ was arbitrary, on letting $\varepsilon \to 0^+$ it follows that
\begin{align} \label{n over log product}
\lim_{n \to \infty} \frac{n}{\log \left( \prod \limits_{k=1}^n u_k \right)} = \frac{1}{\log B}.
\end{align}
It is easy to see that the convergence \eqref{n over log product} still holds if $B = \infty$.

Let $\eta$ be defined as in the statement of the lemma. For the upper bound, by the right-hand inequality of \eqref{rough dimension of E ln rn}, the left-hand inequality of \eqref{bounds for factorial}, and \eqref{n over log product}, we obtain
\begin{align*}
\hdim \mathcal{E}^*
&\leq \liminf_{n \to \infty} \frac{\log \left( \prod \limits_{k=1}^n \Delta_k \right)}{\log \left( \dfrac{r_{n+1}}{\Delta_{n+1}} \prod \limits_{k=1}^{n+1} l_k \right)} 
= \liminf_{n \to \infty} \frac{\log \left( \prod \limits_{k=1}^n u_k \right)}{\log \left( \dfrac{(n+2)u_{n+1}}{u_{n+1}} \prod \limits_{k=1}^{n+1} (ku_k) \right)} \\
&= \liminf_{n \to \infty} \frac{\log \left( \prod \limits_{k=1}^n u_k \right)}{\log \left( \prod \limits_{k=1}^{n+1} u_k \right) + \log (n+2)!} 
\leq \liminf_{n \to \infty} \frac{\log \left( \prod \limits_{k=1}^n u_k \right)}{\log \left( \prod \limits_{k=1}^{n+1} u_k \right) + (n+2) \log (n+2) - (n +1)} \\
&= \bigg[ \limsup_{n \to \infty} \bigg( 1 + \frac{\log u_{n+1} + n \log n - n}{\log \left( \prod \limits_{k=1}^n u_k \right)} + \frac{n \log ( 1 + 2 n^{-1} ) + 2 \log (n+2) - 1}{\log \left( \prod \limits_{k=1}^n u_k \right)}\bigg) \bigg]^{-1} \\
&= \left( 1+\eta - \frac{1}{\log B} \right)^{-1}.
\end{align*}

On the other hand, for the lower bound, by the left-hand inequality of \eqref{rough dimension of E ln rn}, the right-hand inequality of \eqref{bounds for factorial}, and \eqref{n over log product}, we obtain
\begin{align*}
\hdim \mathcal{E}^*
&\geq \liminf_{n \to \infty} \frac{\log \left( \prod \limits_{k=1}^n \Delta_k \right)}{\log \left( \dfrac{r_{n+1} r_{n+2}}{\Delta_{n+1} \Delta_{n+2}} \prod \limits_{k=1}^{n+1} r_k \right)} 
= \liminf_{n \to \infty} \frac{\log \left( \prod \limits_{k=1}^n u_k \right)}{\log \left( \dfrac{(n+2)u_{n+1} (n+3)u_{n+2}}{u_{n+1} u_{n+2}} \prod \limits_{k=1}^{n+1} ((k+1)u_k) \right)} \\
&\geq \liminf_{n \to \infty} \frac{\log \left( \prod \limits_{k=1}^n u_k \right)}{\log \left( \prod \limits_{k=1}^{n+1} u_k \right) + \log (n+4)!} 
\geq \liminf_{n \to \infty} \frac{\log \left( \prod \limits_{k=1}^n u_k \right)}{\log \left( \prod \limits_{k=1}^{n+1} u_k \right) + (n+5) \log (n+4) - (n+3)} \\
&= \bigg[ \limsup_{n \to \infty} \bigg( 1 + \frac{\log u_{n+1} + n \log n - n}{\log \left( \prod \limits_{k=1}^n u_k \right)} + \frac{n \log ( 1 + 4 n^{-1}) + 5 \log (n+4) - 3}{\log \left( \prod \limits_{k=1}^n u_k \right)} \bigg) \bigg]^{-1} \\
&= \left( 1+\eta - \frac{1}{\log B} \right)^{-1}.
\end{align*}
Combining the upper and lower bounds, we deduce that $\hdim \mathcal{E}^* = (1+\eta - (\log B)^{-1})^{-1}$.

For the desired formula for $\hdim \mathcal{E}^*$, first notice that if $B<\infty$, then $(n \log n) / \log \left( \prod_{k=1}^n u_k \right) \to (\log B)^{-1} \cdot \infty = \infty$ as $n \to \infty$ by \eqref{n over log product}, which implies that $\eta = \infty$, and so $1+\eta -(\log B)^{-1} = \infty = 1+\eta$ by the convention $c + \infty = \infty$ for any $c \in \mathbb{R}$. If $B = \infty$, then $1+\eta-(\log B)^{-1} = 1+\eta$ by the convention $\infty^{-1} = 0$. This completes the proof.
\end{proof}

The obvious upper bound for the Hausdorff dimension of subsets of $(0,1]$ is $1$. Hence, to show that a subset of $(0,1]$ is of Hausdorff dimension $1$ by using \eqref{rough dimension of E ln rn}, it is sufficient to consider the lower bound. The lemma below provides an instance. It will be used in the proofs of the results related to the CLT and the LIL.

\begin{lemma} \label{dimension of E check}
Let $f \colon \mathbb{N} \to \mathbb{R}$ be a function defined by $f(n) \coloneqq n + \beta \psi (n)$, where $\beta \in \mathbb{R}$ is an arbitrary constant and $\psi \colon \mathbb{N} \to (0, \infty)$ is a function satisfying conditions (\ref{theorem C psi beta 1}) and (\ref{theorem C psi beta 2}) in Theorem \ref{theorem C psi beta}. Define $L_n \coloneqq e^{f(n)}$ and $R_n \coloneqq ( 1 + \psi(n) n^{-1} ) L_n$ for each $n \in \mathbb{N}$. Then there exists $K \in \mathbb{N}$ such that the sequences $(l_n)_{n \in \mathbb{N}}$ and $(r_n)_{n \in \mathbb{N}}$ defined by
\begin{align} \label{ln rn definition}
l_n \coloneqq
\begin{cases}
2n, &\text{if } 1 \leq n \leq K; \\
L_n, &\text{if } n > K,
\end{cases}
\quad \text{and} \quad
r_n \coloneqq
\begin{cases}
2(n+1), &\text{if } 1 \leq n \leq K; \\
R_n, &\text{if } n > K,
\end{cases}
\end{align}
respectively, satisfy conditions (\ref{rough dimension of E lemma 1})--(\ref{rough dimension of E lemma 3}) in Lemma \ref{rough dimension of E lemma}. Furthermore, $\hdim {\mathcal{E}} = 1$, where the set ${\mathcal{E}} = \mathcal{E}((l_n)_{n \in \mathbb{N}}, (r_n)_{n \in \mathbb{N}})$ is defined as in \eqref{definition of E ln rn}.
\end{lemma}

\begin{proof}
Since $\psi (n+1) - \psi (n) \to 0$ as $n \to \infty$ by condition (\ref{theorem C psi beta 1}) imposed on $\psi$, for any $\varepsilon > 0$ there exists an $M = M(\varepsilon) \in \mathbb{N}$ such that $- \varepsilon < \psi (n+1) - \psi (n) < \varepsilon$ for all $n \geq M$. Then
\[
\psi (M) - (n-M) \varepsilon \leq \psi (n) \leq \psi (M) + (n-M) \varepsilon
\]
for all $n \geq M$, and further, by the definition of $f \colon \mathbb{N} \to \mathbb{R}$,
\[
f(M) + (n-M) (1 - |\beta| \varepsilon) \leq f(n) \leq f(M) + (n-M) (1 + |\beta| \varepsilon)
\]
for all $n \geq M$. Hence
\[
0 \leq \liminf_{n \to \infty} \frac{\psi (n)}{n} \leq \limsup_{n \to \infty} \frac{\psi (n)}{n} \leq \varepsilon
\]
and
\[
\frac{1}{2} (1 - |\beta| \varepsilon) \leq \liminf_{n \to \infty} \sum_{k=1}^n \frac{f(k)}{n^2} \leq \limsup_{n \to \infty} \sum_{k=1}^n \frac{f(k)}{n^2} \leq \frac{1}{2} (1 + |\beta| \varepsilon).
\]
Since $\varepsilon>0$ is arbitrary, it follows that 
\begin{align} \label{psi n over n and sum of f over n squared}
\lim_{n \to \infty} \frac{\psi (n)}{n} = 0 \quad \text{and} \quad \lim_{n \to \infty} \sum_{k=1}^n \frac{f(k)}{n^2} = \frac{1}{2}.
\end{align}

If $\beta \geq 0$, then $L_n \geq e^n$ for each $n \in \mathbb{N}$. If $\beta < 0$, by the first convergence in \eqref{psi n over n and sum of f over n squared} we can find an $M' \in \mathbb{N}$ such that $0 \leq \psi (n) n^{-1} < - (2 \beta)^{-1}$, and so $L_n \geq e^{n/2}$ for all $n \geq M'$. Hence, in either case, we can find $K_1 \in \mathbb{N}$ such that $L_{n+1} \geq 2(n+1)$ for all $n \geq K_1$. Clearly, $L_n \to \infty$ as $n \to \infty$. We further observe the following:
\begin{enumerate}[label=\upshape(\arabic*), ref=\arabic*, leftmargin=*, widest=3]
\item \label{observation 1}
By the preceding discussion, for all $n \geq M'$, we have $R_n-L_n = \psi (n) n^{-1} L_n \geq \psi (n) n^{-1} e^{n/2}$. We know that $e^{n/4} \psi (n) > 1$ for all sufficiently large $n$ by (\ref{theorem C psi beta 2}). Then $R_n-L_n > e^{n/4} n^{-1}$ for all sufficiently large $n$, so that $R_n-L_n \to \infty$ as $n \to \infty$. Now, we can find a $K_2 \in \mathbb{N}$ such that $R_n-L_n \geq 2$ for all $n \geq K_2$.
\item \label{observation 2}
Since $2 L_{n+1} - R_{n+1} = ( 1 - \psi (n+1) (n+1)^{-1} ) L_{n+1} \to (1-0) \cdot \infty = \infty$ as $n \to \infty$ by \eqref{psi n over n and sum of f over n squared}, there exists $K_3 \in \mathbb{N}$ such that $2 L_{n+1} - R_{n+1} \geq 1$ for all $n \geq K_3$.
\item \label{observation 3}
Since
\[
L_{n+1} - R_n = ( e^{1 + \beta [\psi (n+1) - \psi (n)]} - 1 - \psi (n) n^{-1} ) L_n \to (e^{1+\beta \cdot 0}-1-0) \cdot \infty = \infty
\]
as $n \to \infty$ by (\ref{theorem C psi beta 1}) and \eqref{psi n over n and sum of f over n squared}, there exists $K_4 \in \mathbb{N}$ such that $R_n \leq L_{n+1}$ for all $n \geq K_4$.
\end{enumerate}
Put $K \coloneqq \max \{ K_1, K_2, K_3, K_4 \}$. Define sequences $(l_n)_{n \in \mathbb{N}}$ and $(r_n)_{n \in \mathbb{N}}$ as in \eqref{ln rn definition}. For $1 \leq n \leq K$, we have
\begin{enumerate}[label=\upshape(\arabic*{$'$}), ref=\arabic*{$'$}, leftmargin=*, widest=3]
\item \label{observation 1 prime}
$\Delta_n \equiv 2$,
\item \label{observation 2 prime}
$2l_{n+1} - r_{n+1} = 2 \cdot 2(n+1) - 2(n+2) = 2n \geq 1$ if $n \neq K$,
\item \label{observation 3 prime}
$r_n = l_{n+1} = 2(n+1)$ if $n \neq K$, and $r_K = 2(K+1) \leq L_{K+1} = l_{K+1}$.
\end{enumerate}
Combining (\ref{observation 1})--(\ref{observation 3}) and (\ref{observation 1 prime})--(\ref{observation 3 prime}), we conclude that the sequences $(l_n)_{n \in \mathbb{N}}$ and $(r_n)_{n \in \mathbb{N}}$ satisfy conditions (\ref{rough dimension of E lemma 1})--(\ref{rough dimension of E lemma 3}) in Lemma \ref{rough dimension of E lemma}. 

Consequently, we may use \eqref{rough dimension of E ln rn} in Lemma \ref{rough dimension of E lemma} to prove the last assertion. Let $p \in (0,\infty)$ be arbitrary. To calculate the lower bound in \eqref{rough dimension of E ln rn}, we write for each $n>K$,
\[
g(n) \coloneqq \frac{\log \left( \prod \limits_{k=1}^n \Delta_k \right)}{\log \left( \dfrac{r_{n+1} r_{n+2}}{\Delta_{n+1} \Delta_{n+2}} \prod \limits_{k=1}^{n+1} r_k \right)}
= \frac{\log \left( \prod \limits_{k=1}^K \Delta_k \cdot \prod \limits_{k=K+1}^n \Delta_k \right)}{\log \left( \prod \limits_{k=1}^{K} r_k \cdot \dfrac{r_{n+1} r_{n+2}}{\Delta_{n+1} \Delta_{n+2}} \prod \limits_{k=K+1}^{n+1} r_k \right)}.
\]
Let $C$ and $D$ be constants given by
\begin{align*}
C &\coloneqq - \log \left[ \prod_{k=1}^K \left( \frac{\psi(k)}{k} e^{f(k)} \right) \right] + K \log 2, \\
D &\coloneqq - \log \left[ \prod_{k=1}^K \left( \left( 1 + \frac{\psi(k)}{k} \right) e^{f(k)} \right) \right] + K \log 2 + \log (K+1)! ,
\end{align*}
respectively. Then for each $n>K$, we have
\begin{align*}
g(n) 
&= \frac{C+\log \left( \prod \limits_{k=1}^n \left( \dfrac{\psi (k)}{k} e^{f(k)} \right) \right)}{D+\log \left( \dfrac{\left( 1 + \dfrac{\psi (n+1)}{n+1} \right) L_{n+1} \left( 1 + \dfrac{\psi (n+2)}{n+2} \right) L_{n+2}}{\dfrac{\psi(n+1)}{n+1} L_{n+1} \dfrac{\psi(n+2)}{n+2} L_{n+2}}  \prod \limits_{k=1}^{n+1} \left( \left( 1 + \dfrac{\psi (k)}{k} \right) e^{f(k)} \right) \right)} \\
&= \frac{C+\sum \limits_{k=1}^n f(k) + \sum \limits_{k=1}^n \log \psi (k) - \log n!}{D+\log \left[ \left( \dfrac{n+1}{\psi(n+1)} +1 \right) \left( \dfrac{n+2}{\psi(n+2)} +1 \right) \right] + \sum \limits_{k=1}^{n+1} f(k) + \sum \limits_{k=1}^{n+1} \log \left( 1 + \dfrac{\psi (k)}{k} \right)}.
\end{align*}
Now, by condition (\ref{theorem C psi beta 2}) on $\psi$ and the first convergence in \eqref{psi n over n and sum of f over n squared}, we can find an integer $M'' \geq K$ such that $e^{pn} \psi (n) > 1$ and $n > \psi (n)$ for all $n > M''$. Hence, for all $n > M''$, 
\[
\log \psi (n) > - p n, \quad \frac{n}{\psi (n)} + 1 < \frac{2n}{\psi (n)} < 2 n e^{pn}, \quad \log \left( 1 + \frac{\psi (n)}{n} \right) < \log 2,
\]
and it follows by \eqref{bounds for factorial} that
\begin{align*}
g(n)
&\geq \frac{C+\sum \limits_{k=1}^n f(k) + \sum \limits_{k=1}^{M''} \log \psi (k) - \sum \limits_{k=M''+1}^n pk - [(n+1) \log n - (n-1)]}{D+\log [2(n+1)e^{p(n+1)} \cdot 2(n+2)e^{p(n+2)}] + \sum \limits_{k=1}^{n+1} f(k) + \sum \limits_{k=1}^{M''} \log \left( 1 + \dfrac{\psi(k)}{k} \right) + \sum \limits_{k=M''+1}^{n+1} \log 2} \\
&= \Bigg[\frac{C}{n^2} + \sum \limits_{k=1}^n \dfrac{f(k)}{n^2} + \sum \limits_{k=1}^{M''} \dfrac{\log \psi (k)}{n^2} - \dfrac{p(n+M''+1)(n-M'')}{2n^2} - \dfrac{(n+1) \log n - (n-1)}{n^2} \Bigg] \\
&\hspace{1cm} \Bigg/
\Bigg[ \frac{D}{n^2} + \dfrac{\log [4(n+1)(n+2)e^{p(2n+3)}]}{n^2} + \dfrac{(n+1)^2}{n^2} \sum \limits_{k=1}^{n+1} \dfrac{f(k)}{(n+1)^2}  \\
&\hspace{5cm} + \dfrac{1}{n^2} \sum \limits_{k=1}^{M''} \log \left( 1 + \dfrac{\psi(k)}{k} \right) + \dfrac{(n-M''+1) \log 2}{n^2} \Bigg]
\end{align*}
for all $n > M''$. Notice that the whole expression in the three lines above converges to
\begin{align*}
\frac{0 + (1/2) + 0 - (p/2) - 0}{0 + 0 + (1)(1/2) + 0 + 0} = 1-p
\end{align*}
as $n \to \infty$ by the second convergence in \eqref{psi n over n and sum of f over n squared}. Therefore, by the left-hand inequality of \eqref{rough dimension of E ln rn}, we conclude that
\begin{align*}
\hdim \mathcal{E} \geq \liminf_{n \to \infty} g(n) \geq 1-p.
\end{align*}
Since $p \in (0, \infty)$ was arbitrary, letting $p \to 0^+$ we conclude that $\hdim \mathcal{E} \geq 1$. It is obvious that $\hdim \mathcal{E} \leq 1$, and this completes the proof.
\end{proof}

The following lemma is an analogue of \cite[Lemma 2.1]{She10}, where the Engel expansion is discussed. This lemma tells us in particular that neither adding (or shifting to the right) nor removing (or shifting to the left) finitely many initial Pierce expansion digits alters the Hausdorff dimension.

\begin{lemma} \label{shifting of digits}
Let $n \in \mathbb{N}$ and $F \subseteq (0,1]$. For each $\sigma \coloneqq (\sigma_k)_{k=1}^n \in \mathbb{N}^n$, define a map $g_\sigma \colon F \to \mathbb{R}$ by
\[
g_\sigma (x) \coloneqq \varphi_n (\sigma) + (-1)^n \left( \prod_{k=1}^n \frac{1}{\sigma_k} \right)  x
\]
for each $x \in F$. Put
\[
\mathcal{F} \coloneqq \bigcup_{\sigma \in \mathbb{N}^n} g_\sigma (F).
\]
Then $\hdim \mathcal{F} = \hdim F$.
\end{lemma}

\begin{proof}
It is clear that $g_\sigma$ is well-defined for each $\sigma \in \mathbb{N}^n$. It is also clear that $g_\sigma \colon F \to g_\sigma (F)$ is a bi-Lipschitz mapping since $g_\sigma$ is linear on $F \subseteq (0,1]$. Indeed, for any $x,y \in F$,
\begin{align*}
| g_\sigma (x) - g_\sigma (y) | = \left( \prod_{k=1}^n \frac{1}{\sigma_k} \right) |x-y|.
\end{align*}
It follows that $\hdim F = \hdim g_\sigma (F)$ for any $\sigma \in \mathbb{N}^n$. Note that $\mathbb{N}^n$ is countable. Therefore, by countable stability of the Hausdorff dimension, we deduce that
\[
\hdim \mathcal{F} = \hdim \bigcup_{\sigma \in \mathbb{N}^n} g_\sigma (F) = \sup_{\sigma \in \mathbb{N}^n} \hdim g_\sigma (F) = \hdim F,
\]
as desired.
\end{proof}

\begin{remark} \label{remark to shifting of digits lemma}
If $x = \langle d_1(x), d_2(x), \dotsc \rangle_P \in F$, then the Pierce expansion of $g_\sigma (x)$, where $\sigma \coloneqq (\sigma_k)_{k=1}^n \in \mathbb{N}^n$, is not necessarily given by
\[
\langle \sigma_1, \dotsc, \sigma_n, d_1(x), d_2(x), \dotsc \rangle_P
\]
since the inequalities $\sigma_1 < \dotsb < \sigma_n < d_1(x)$ might not be true (see Propositions \ref{digit condition lemma} and \ref{strict increasing condition}). It might even be possible that $g_\sigma (x)$ does not lie in $(0,1]$. However, for $y = \langle d_1(y), d_2(y), \dotsc \rangle_P \in \mathbb{I}$, the Pierce expansion of $x \coloneqq g_{\tau}^{-1}(y)$, where $\tau \coloneqq (d_k(y))_{k=1}^n \in \Sigma_n \subseteq \mathbb{N}^n$, is equal to
\[
\langle d_{n+1}(y), d_{n+2}(y), \dotsc \rangle_P,
\]
in which case $y = g_\tau (x) \in \bigcup_{\sigma \in \mathbb{N}^n} \{ g_\sigma (x) \}$.
\end{remark}

The following lemma is useful in investigating the Hausdorff dimensions of sets defined in terms of two consecutive digits.

\begin{lemma} \label{preservation of dimension}
Let $h_i \colon \mathbb{N} \to \mathbb{R}$ be any functions for $i \in \{ 1, 2, 3 \}$. Let $K \in \mathbb{N}$. Then, the following hold:
\begin{enumerate}[label=\upshape(\roman*), ref=\roman*, leftmargin=*, widest=ii]
\item \label{preservation of dimension 1}
For each integer $m \geq K$, define 
\[
\mathcal{S} (m, h_1, h_2) \coloneqq \{ x \in (0,1] : h_1 (d_n(x)) < h_2 (d_{n+1}(x)) \text{ for all } n \geq m \}.
\]
Then, for each $m \geq K$,
\[
\hdim \mathcal{S} (m, h_1, h_2) = \hdim \mathcal{S} (K, h_1, h_2).
\]
\item \label{preservation of dimension 2}
For each integer $m \geq K$, define 
\[
\mathcal{S} (m, h_1, h_2, h_3) \coloneqq \{ x \in (0,1] : h_1 (d_n(x)) < h_2 (d_{n+1}(x)) \leq h_3 (d_n(x)) \text{ for all } n \geq m \}.
\]
Then, for each $m \geq K$,
\[
\hdim \mathcal{S} (m, h_1, h_2, h_3) = \hdim \mathcal{S} (K, h_1, h_2, h_3).
\]
\end{enumerate}
\end{lemma}

\begin{proof} 
Since the proofs of two parts are virtually identical, we prove part (\ref{preservation of dimension 2}) only. For each integer $m \geq K$, we write $\mathcal{S}(m) \coloneqq \mathcal{S} (m, h_1, h_2, h_3)$ for brevity. For any $m \geq K$, since $\mathcal{S} (K) \subseteq \mathcal{S} (m)$ by definition, we have $\hdim \mathcal{S} (K) \leq \hdim \mathcal{S} (m)$ by monotonicity of the Hausdorff dimension. To prove the reverse inequality, take $F \coloneqq \mathcal{S} (K) \subseteq (0,1]$ in Lemma \ref{shifting of digits}. Then, for any $m > K$,
\[
\mathcal{F} \coloneqq \bigcup_{\sigma \in \mathbb{N}^{m-K}} g_\sigma (F) \supseteq \mathcal{S} (m).
\]
To see this, let $m > K$ and $y \in \mathcal{S} (m)$. By definition of $\mathcal{S}(m)$, it is clear that $y \in \mathbb{I}$. As discussed in Remark \ref{remark to shifting of digits lemma}, we may write $y = g_\tau (x)$, where $\tau \coloneqq (d_j(y))_{j=1}^{m-K} \in \Sigma_{m-K} \subseteq \mathbb{N}^{m-K}$ and
\[
x \coloneqq \langle d_{m-K+1}(y), d_{m-K+2}(y), \dotsc \rangle_P.
\]
Notice that $d_j (x) = d_{m-K+j} (y)$ for all $j \in \mathbb{N}$ by the uniqueness of the Pierce expansion. Since $y \in \mathcal{S} (m)$, we have
\[
h_1 (d_n(y)) < h_2 (d_{n+1}(y)) \leq h_3 (d_n(y)) \quad \text{for all } n \geq m
\]
by definition. Hence, for all $n \geq K$,
\[
h_1 (d_n(x)) < h_2 (d_{n+1}(x)) \leq h_3 (d_n(x)).
\]
This shows that $x \in \mathcal{S} (K) = F$ and $y = g_\tau (x) \in g_\tau (F)$. Hence, $y \in \mathcal{F}$, and we conclude that $\mathcal{S} (m) \subseteq \mathcal{F}$. But then monotonicity of the Hausdorff dimension and Lemma \ref{shifting of digits} tell us that
\[
\hdim \mathcal{S} (m) \leq \hdim \mathcal{F} = \hdim F = \hdim \mathcal{S} (K),
\]
and the statement is proved.
\end{proof}

\section{Proofs of the results} \label{proofs of the results}

\subsection{The growth speed of the digits and the law of large numbers}

This subsection is devoted to proving the first two main results of this paper and their corollaries. 

\subsubsection{Proof of Theorem \ref{theorem E phi}} \label{subsubsection theorem E phi}

Throughout Section \ref{subsubsection theorem E phi}, let $\phi \colon\mathbb{N} \to (0, \infty)$, $\gamma$, $\xi$, and $E(\phi)$ be as described in Theorem \ref{theorem E phi}.

\begin{lemma} \label{E phi lemma 1} 
If $\limsup_{n \to \infty} \phi (n)/\log n < 1$, then $E(\phi) = \varnothing$ and thus $\hdim E(\phi) = 0$.
\end{lemma}

\begin{proof}
Put $\gamma' \coloneqq \limsup_{n \to \infty} \phi (n)/\log n$. By the hypothesis, $\gamma' \in [0,1)$ since $\phi$ is positive real-valued. Take $\varepsilon > 0$ so small that $\gamma' + \varepsilon < 1$. By definition of $\limsup$, we can find an $M = M(\varepsilon) \in \mathbb{N}$ such that $\phi (n) < (\gamma' + \varepsilon) \log n$ for all $n > M$. Now, for any $x \in \mathbb{I}$, if $n > M$, then
\[
\frac{\log d_n(x)}{\phi (n)} > \frac{\log d_n(x)}{(\gamma' + \varepsilon) \log n} \geq \frac{\log n}{(\gamma' + \varepsilon) \log n} = \frac{1}{\gamma' + \varepsilon},
\]
where the second inequality follows from Proposition \ref{digit condition lemma}(\ref{digit condition lemma 2}). Then $\liminf_{n \to \infty} \log d_n(x) / \phi (n) \geq (\gamma' + \varepsilon)^{-1} > 1$, and this shows that $E(\phi) = \varnothing$. Thus, $\hdim E(\phi) = 0$.
\end{proof}

\begin{lemma} \label{n log n over sum lemma}
If $\gamma \in [1,\infty]$, then
\begin{align*} 
\lim_{n \to \infty} \frac{n \log n}{\sum \limits_{k=1}^n \phi (k)} = \gamma^{-1}.
\end{align*}
\end{lemma}

\begin{proof}
Let $\gamma \in [1, \infty)$. Fix $\varepsilon \in (0, \gamma)$. Since ${\phi(n)}/{\log n} \to \gamma$ as $n \to \infty$, we can find an $M = M(\varepsilon) \in \mathbb{N}$ such that $(\gamma-\varepsilon) \log n < \phi (n) < (\gamma+\varepsilon) \log n$ for all $n > M$. Put $A \coloneqq \sum_{k=1}^M \phi (k)$. Then, on one hand, by the left-hand inequality of \eqref{bounds for sum of log}, for any $n>M$ we have
\begin{align*}
\frac{n \log n}{\sum \limits_{k=1}^n \phi (k)}
\leq \frac{n \log n}{A + (\gamma-\varepsilon) \sum \limits_{k=M+1}^n \log k} 
\leq \frac{n \log n}{A + \displaystyle (\gamma-\varepsilon) \int_M^n \log x \, dx}.
\end{align*}
On the other hand, by the right-hand inequality of \eqref{bounds for sum of log}, for any $n>M$ we have
\begin{align*}
\frac{n \log n}{\sum \limits_{k=1}^n \phi (k)}
\geq \frac{n \log n}{A + (\gamma+\varepsilon) \sum \limits_{k=M+1}^n \log k} 
\geq \frac{n \log n}{A + \displaystyle (\gamma+\varepsilon) \int_{M+1}^{n+1} \log x \, dx}.
\end{align*}
Thus
\[
\frac{1}{\gamma+\varepsilon} \leq \liminf_{n \to \infty} \frac{n \log n}{\sum \limits_{k=1}^n \phi (k)} \leq \limsup_{n \to \infty} \frac{n \log n}{\sum \limits_{k=1}^n \phi (k)} \leq \frac{1}{\gamma-\varepsilon}.
\]
Letting $\varepsilon \to 0^+$, we obtain the desired convergence.

Now, assume $\gamma = \infty$. Fix $B>0$. Since ${\phi(n)}/{\log n} \to \infty$ as $n \to \infty$, we can find an $M' = M'(B) \in \mathbb{N}$ such that $\phi (n) > B \log n$ for all $n>M'$. By an argument similar to the one in the preceding paragraph,
\[
0 \leq \liminf_{n \to \infty} \frac{n \log n}{\sum \limits_{k=1}^n \phi (k)} \leq \limsup_{n \to \infty} \frac{n \log n}{\sum \limits_{k=1}^n \phi (k)} \leq \frac{1}{B}.
\]
Letting $B \to \infty$, we obtain the desired convergence.
\end{proof}

\begin{lemma} \label{useful lemma}
If $\gamma \in [1, \infty)$, then $\xi = \lim_{n \to \infty} ( \phi (n+1) \big/ \sum_{k=1}^n \phi (k) ) = 0$.
\end{lemma}

\begin{proof}
Let $\gamma \in [1, \infty)$. Then since ${\phi(n)}/{\log n} \to \gamma$ as $n \to \infty$, we find by Lemma \ref{n log n over sum lemma} that
\begin{align*} 
\frac{\phi(n+1)}{\sum \limits_{k=1}^n \phi (k)} = \frac{n \log n}{\sum \limits_{k=1}^n \phi (k)} \cdot \frac{\phi (n+1)}{\log (n+1)} \cdot \frac{\log (n+1)}{\log n} \cdot \frac{1}{n} \to \gamma^{-1} \cdot \gamma \cdot 1 \cdot 0 = 0
\end{align*}
as $n \to \infty$.
\end{proof}

\begin{lemma} \label{E phi lower bound}
If $\gamma \in [1,\infty]$, then
\[
\hdim E(\phi) \geq \frac{1-\gamma^{-1}}{1 + (1-\gamma^{-1})\xi}.
\]
\end{lemma}

\begin{proof}
If $\gamma = 1$, then $\xi=0$ by Lemma \ref{useful lemma}, and hence the inequality holds trivially. Let $\gamma \in (1, \infty]$. Define a sequence $(u_n)_{n \in \mathbb{N}}$ by $u_n \coloneqq e^{(1 - \gamma^{-1} ) \phi (n)+1}$ for each $n \in \mathbb{N}$. Clearly, $2 \leq u_n \leq u_{n+1}$ for all $n \in \mathbb{N}$. By Lemma \ref{n log n over sum lemma},
\begin{align*}
\eta 
&\coloneqq  \limsup_{n \to \infty} \frac{n \log n + \log u_{n+1}}{\log \left( \prod \limits_{k=1}^n u_k \right)} 
= \limsup_{n \to \infty} \frac{n \log n + (1 - \gamma^{-1} ) \phi (n+1)+1}{n + (1 - \gamma^{-1} ) \sum \limits_{k=1}^n \phi (k)} \\
&= \limsup_{n \to \infty} \frac{\left( n \log n \Big/ \sum \limits_{k=1}^n \phi (k) \right) + (1 - \gamma^{-1} ) \left( \phi (n+1) \Big/ \sum \limits_{k=1}^n \phi (k) \right) + \left( n \log n \Big/ \sum \limits_{k=1}^n \phi (k) \right) (1 / n \log n)}{\left( n \log n \Big/ \sum \limits_{k=1}^n \phi (k) \right) (1 / \log n) + (1 - \gamma^{-1} )} \\
&= \frac{\gamma^{-1} + ( 1- \gamma^{-1} ) \xi + (\gamma^{-1})(0)}{(\gamma^{-1})(0) + (1- \gamma^{-1})} = \frac{\gamma^{-1}+ ( 1- \gamma^{-1} ) \xi}{1-\gamma^{-1}}.
\end{align*}
It follows from Lemma \ref{dimension of E star} that 
\[
\hdim \mathcal{E}^* = \left( 1 + \frac{\gamma^{-1} + ( 1- \gamma^{-1} ) \xi}{1 - \gamma^{-1}} \right)^{-1} = \frac{1-\gamma^{-1}}{1 + ( 1- \gamma^{-1} ) \xi},
\]
where $\mathcal{E}^*$ is defined as in \eqref{definition of E un}, i.e.,
\[
\mathcal{E}^* = \left\{ x \in (0,1] : n e^{(1-\gamma^{-1}) \phi (n)+1} < d_n(x) \leq (n+1) e^{(1-\gamma^{-1}) \phi (n)+1} \text{ for all } n \in \mathbb{N} \right\}.
\] 
Then, for any $x \in \mathcal{E}^*$,
\[
1 - \gamma^{-1} + \frac{\log n+1}{\phi (n)} < \frac{\log d_n(x)}{\phi (n)} \leq 1 - \gamma^{-1} + \frac{\log (n+1) + 1}{\phi (n)} \quad \text{for all } n \in \mathbb{N}.
\]
Note that $(\log n)/{\phi (n)} \to \gamma^{-1}$ and ${1}/{\phi (n)} \to 0$ as $n \to \infty$ by the assumptions on $\phi \colon \mathbb{N} \to (0,\infty)$. Hence, $(\log d_n(x)) / \phi (n) \to 1$ as $n \to \infty$, and this shows that $x \in E(\phi)$. Thus, we have the inclusion $E(\phi) \supseteq \mathcal{E}^*$, and so, using monotonicity of the Hausdorff dimension, we find that $\hdim E(\phi) \geq \hdim \mathcal{E}^* = (1 - \gamma^{-1})/[1 + ( 1- \gamma^{-1} ) \xi]$.
\end{proof}

Let
\begin{align} \label{definition of theta}
\theta (n) \coloneqq \frac{n \phi (n)}{\sum \limits_{k=1}^n \phi (k)} \quad \text{for each $n \in \mathbb{N}$} \quad \text{and} \quad \theta \coloneqq \liminf_{n \to \infty} \theta (n).
\end{align}

\begin{lemma} \label{n phi over sum} 
We have $\theta \geq 1$. In particular, if $\gamma \in [1, \infty)$, then $\theta = \lim_{n \to \infty} \theta (n) = 1$.
\end{lemma}

\begin{proof}
Since $\phi \colon \mathbb{N} \to (0,\infty)$ is non-decreasing and positive real-valued, we have $n \phi (n) \geq \sum_{k=1}^n \phi (k) > 0$, and so $\theta \geq 1$.

The second assertion follows from Lemma \ref{n log n over sum lemma}. Indeed, if $\gamma \in [1, \infty)$, then
\[
\theta (n) = \frac{n \phi (n)}{\sum \limits_{k=1}^n \phi (k)} = \frac{n \log n}{\sum \limits_{k=1}^n \phi (k)} \cdot \frac{\phi (n)}{\log n} \to \gamma^{-1} \cdot \gamma = 1
\]
as $n \to \infty$.
\end{proof}

The following lemma and its proof are motivated by \cite[Theorems 3.2 and 3.6]{SW20} and \cite[Theorem 3.1]{SW21} in which an analogous set for Engel expansions is discussed.

\begin{lemma} \label{E phi upper bound} 
If $\gamma \in [1,\infty]$, then
\[
\hdim E(\phi) \leq \min \{ (1+\xi)^{-1}, \theta - \gamma^{-1} \}.
\]
\end{lemma}

\begin{proof}
Let $\gamma \in [1,\infty]$. Fix $\varepsilon \in (0,1)$. For each $m \in \mathbb{N}$, let
\begin{align*}
E_\varepsilon (m, \phi) &\coloneqq \left\{ x \in (0,1] : e^{(1-\varepsilon) \phi (k)} < d_k (x) \leq e^{(1+\varepsilon) \phi (k)} \text{ for all } k \geq m \right\}.
\end{align*}
Note that since $\phi (k)$ is increasing and tending to infinity as $k \to \infty$, so is the difference of bounds for $d_k(x)$ in the definition of $E_\varepsilon (m, \phi)$, that is $e^{(1+\varepsilon) \phi (k)} - e^{(1-\varepsilon) \phi (k)}$ is monotonically increasing to $\infty$ as $k \to \infty$. Hence, there exists $K_1 \in \mathbb{N}$ such that
\[
\lfloor e^{(1-\varepsilon) \phi (n)} \rfloor + 1 \leq \lfloor e^{(1+\varepsilon) \phi (n)} \rfloor \quad \text{for all } n \geq K_1.
\]
Further, since ${\phi (n)}/{\log n} \to \gamma$ as $n \to \infty$, there exists $K_2 \in \mathbb{N}$ such that
\[
n \leq \lfloor e^{(1+\varepsilon) \phi (n)} \rfloor \quad \text{for all } n \geq K_2.
\]
Put $K \coloneqq \max \{ K_1, K_2 \}$.

Let
\[
E_\varepsilon ( \phi) \coloneqq \bigcup_{m \geq K} E_\varepsilon (m, \phi).
\]
Since $E_\varepsilon (m, \phi) \subseteq E_\varepsilon (m+1, \phi)$ for all $m \in \mathbb{N}$ by definition, it is clear that $E(\phi) \subseteq E_\varepsilon ( \phi)$. Now, by monotonicity and countable stability of the Hausdorff dimension, it follows that
\begin{align} \label{hdim E phi upper bound}
\hdim E(\phi) \leq \hdim E_\varepsilon ( \phi) = \sup_{m \geq K} \hdim E_\varepsilon (m, \phi) .
\end{align}

We show that
\begin{align} \label{sufficient condition}
\hdim E_\varepsilon (m, \phi) \leq \min \left\{ \frac{1+\varepsilon}{1-\varepsilon} \frac{1}{1+\xi}, \frac{(1+\varepsilon)\theta - \gamma^{-1}}{1-\varepsilon} \right\}
\end{align}
for any positive integer $m \geq K$. If $E_\varepsilon (m, \phi)$ is empty, then \eqref{sufficient condition} holds trivially since $(1+\xi)^{-1} \geq 0$ by definition of $\xi$ and since $(1+\varepsilon) \theta - \gamma^{-1} \geq (1+\varepsilon) - \gamma^{-1} \geq 0$ by Lemma \ref{n phi over sum} and the hypothesis. So, we assume that $E_\varepsilon (m, \phi)$ is not empty. 

We shall make use of a symbolic space. Let $\widetilde{\Upsilon}_{m-1} \coloneqq \Sigma_{m-1}$, and, for each $n \geq m$, let
\begin{align*}
\widetilde{\Upsilon}_n \coloneqq \{ (\sigma_k)_{k=1}^n \in \Sigma_n : 
\lfloor e^{(1-\varepsilon) \phi (k)} \rfloor + 1 \leq \sigma_{k} \leq \lfloor e^{(1+\varepsilon) \phi (k)} \rfloor \text{ for all } m \leq k \leq n \}.
\end{align*}
Define
\[
\widetilde{\Upsilon} \coloneqq \bigcup_{n \geq m-1} \widetilde{\Upsilon}_n.
\]
For any $\sigma \coloneqq (\sigma_k)_{k=1}^n \in \widetilde{\Upsilon}_n$, $n \geq m-1$, we define the {\em $n$th level basic interval} $\widetilde{\mathcal{J}} (\sigma)$ by
\[
\widetilde{\mathcal{J}} (\sigma) \coloneqq \bigcup_{j = \lfloor e^{(1-\varepsilon) \phi (n+1)} \rfloor +1}^{\lfloor e^{(1+\varepsilon) \phi (n+1)} \rfloor} 
\overline{\mathcal{I} (\sigma^{(j)})}.
\]
Then, similar to \eqref{equivalent definition of E ln rn}, we have
\[
E_\varepsilon (m, \phi) = \bigcap_{n \geq m-1} \bigcup_{\sigma \in \widetilde{\Upsilon}_n} \widetilde{\mathcal{J}} (\sigma),
\]
It follows that, for any $n \geq m-1$, the collection of all $n$th level basic intervals $\{ \widetilde{\mathcal{J}} (\sigma) : \sigma \in \widetilde{\Upsilon}_n \}$ is a covering of $E_\varepsilon (m, \phi)$. We will obtain the upper bound for the Hausdorff dimension of $E_\varepsilon (m, \phi)$ from this covering. 

For any $n \geq m-1$, by using \eqref{diam I sigma}, a general inequality $\lfloor x \rfloor + 1 > x$ for any $x \in \mathbb{R}$, and $e^{(1-\varepsilon) \phi (k)} < \sigma_k$ for $m \leq k \leq n$, we find that
\begin{align} \label{tilde J sigma}
\diam \widetilde{\mathcal{J}} (\sigma) = \len \widetilde{\mathcal{J}} (\sigma)
&= \sum_{j = \lfloor e^{(1-\varepsilon) \phi (n+1)} \rfloor + 1}^{\lfloor e^{(1+\varepsilon) \phi (n+1)} \rfloor} \left[ \left( \prod_{k=1}^n \frac{1}{\sigma_k} \right) \left( \frac{1}{j} - \frac{1}{j+1} \right) \right] \nonumber \\
&= \left( \prod_{k=1}^n \frac{1}{\sigma_k} \right) \left( \frac{1}{\lfloor e^{(1-\varepsilon) \phi (n+1)} \rfloor + 1} - \frac{1}{\lfloor e^{(1+\varepsilon) \phi (n+1)} \rfloor + 1} \right) \nonumber \\
&\leq \left( \prod_{k=m}^n \frac{1}{e^{(1-\varepsilon) \phi (k)}} \right) \frac{1}{e^{(1-\varepsilon) \phi (n+1)}}
= \prod_{k=m}^{n+1} \frac{1}{e^{(1-\varepsilon) \phi (k)}}.
\end{align}

Now, we estimate $\# \{ \widetilde{\mathcal{J}} (\sigma) : \sigma \in \widetilde{\Upsilon}_n \} = \# \widetilde{\Upsilon}_n$ for large enough $n$. For any $\sigma \in \widetilde{\Upsilon}_n$, write $\sigma \coloneqq (\sigma_k)_{k=1}^n$. On the one hand, for any $\sigma \in \widetilde{\Upsilon}_n$, we have
\begin{enumerate}[label=\upshape(\roman*), leftmargin=*, widest=iii]
\item
$\sigma_m \leq \lfloor e^{(1+\varepsilon) \phi (m)} \rfloor$,
\item
$(\sigma_k)_{k=1}^{m} \in \Sigma_{m}$,
\item
$\lfloor e^{(1-\varepsilon) \phi(k)} \rfloor + 1 \leq \sigma_k \leq \lfloor e^{(1+\varepsilon) \phi (k)} \rfloor$ for $m+1 \leq k \leq n$.
\end{enumerate}
Hence,
\begin{align*} 
\# \widetilde{\Upsilon}_n 
&\leq {\lfloor e^{(1+\varepsilon) \phi(m)} \rfloor \choose m} \prod_{k=m+1}^n \left( \lfloor e^{(1+\varepsilon) \phi(k)} \rfloor - \lfloor e^{(1-\varepsilon) \phi(k)} \rfloor \right) 
\\
&\leq e^{m (1+\varepsilon) \phi(m)} \prod_{k=m+1}^n e^{(1+\varepsilon) \phi (k)} .
\end{align*}
On the other hand, since $\widetilde{\Upsilon}_n \subseteq \Sigma_n$ and $\sigma_n \leq \lfloor e^{(1+\varepsilon) \phi(n)} \rfloor$ for any $\sigma  \in \widetilde{\Upsilon}_n$, we have
\begin{align*} 
\# \widetilde{\Upsilon}_n \leq {\lfloor e^{(1+\varepsilon) \phi(n)} \rfloor \choose n} \leq \frac{e^{n(1+\varepsilon) \phi(n)}}{n!} \leq \frac{e^{n(1+\varepsilon) \phi(n)+(n-1)}}{n^n},
\end{align*}
where we used \eqref{bounds for factorial} for the last inequality. So, we obtain the following bound:
\begin{align} \label{bounds for tilde Dn}
\# \widetilde{\Upsilon}_n \leq \min \left\{ e^{m (1+\varepsilon) \phi(m)} \prod_{k=m+1}^n e^{(1+\varepsilon) \phi (k)}, \frac{e^{n(1+\varepsilon) \phi(n)+(n-1)}}{n^n} \right\}.
\end{align}

We make use of Proposition \ref{box dimension} to find the desired upper bound by taking $\delta_n \coloneqq \max \{ \diam \widetilde{\mathcal{J}} (\sigma) : \sigma \in \widetilde{\Upsilon}_n \}$ for each $n \in \mathbb{N}$. It is clear from \eqref{tilde J sigma} that $\delta_n \to 0$ as $n \to \infty$. Hence, by \eqref{tilde J sigma} and the first bound in \eqref{bounds for tilde Dn},
\begin{align*}
\hdim E_\varepsilon (m, \phi)
&\leq \liminf_{n \to \infty} \frac{\log (\# \{ \widetilde{\mathcal{J}} (\sigma) : \sigma \in \widetilde{\Upsilon}_n \})}{\log \left( 1 / \delta_n \right)} 
= \liminf_{n \to \infty} \frac{\log (\# \widetilde{\Upsilon}_n)}{\log \left( 1 / \delta_n \right)} \\
&\leq \liminf_{n \to \infty} \frac{m (1+\varepsilon) \phi (m) + (1+\varepsilon) \sum \limits_{k=m+1}^n \phi (k)}{(1-\varepsilon) \sum \limits_{k=m}^{n+1} \phi (k)} \\
&= \frac{1+\varepsilon}{1-\varepsilon} \liminf_{n \to \infty} \frac{m \phi (m) + \sum \limits_{k=1}^n \phi (k) - \sum \limits_{k=1}^{m} \phi (k)}{\sum \limits_{k=1}^n \phi (k) - \sum \limits_{k=1}^{m-1} \phi (k) + \phi (n+1)} \\
&= \frac{1+\varepsilon}{1-\varepsilon} \liminf_{n \to \infty} \frac{\left( m \phi (m) \Big/ \sum \limits_{k=1}^n \phi (k) \right) + 1 - \left( \sum \limits_{k=1}^{m} \phi (k) \Big/ \sum \limits_{k=1}^n \phi (k) \right)}{1 - \left( \sum \limits_{k=1}^{m-1} \phi (k) \Big/ \sum \limits_{k=1}^n \phi (k) \right) + \left( \phi (n+1) \Big/ \sum \limits_{k=1}^n \phi (k) \right)} \\
&= \frac{1+\varepsilon}{1-\varepsilon} \frac{0+1-0}{1-0+\xi}
= \frac{1+\varepsilon}{1-\varepsilon} \frac{1}{1+\xi}.
\end{align*}
Here, the second-to-last equality follows from the fact that $\phi (n)$ diverges to $\infty$ as $n \to \infty$ and so does $\sum_{k=1}^n \phi (k)$.

Now, using Proposition \ref{box dimension}, together with \eqref{tilde J sigma} and the second bound in \eqref{bounds for tilde Dn}, we find that
\begin{align*}
\hdim E_\varepsilon (m, \phi)
&\leq \liminf_{n \to \infty} \frac{\log (\# \{ \widetilde{\mathcal{J}} (\sigma) : \sigma \in \widetilde{\Upsilon}_n \})}{\log \left( 1 / \delta_n \right)} 
= \liminf_{n \to \infty} \frac{\log (\# \widetilde{\Upsilon}_n)}{\log \left( 1 / \delta_n \right)} \\
&\leq \liminf_{n \to \infty} \frac{n (1+\varepsilon) \phi (n) +(n-1) - n \log n}{(1-\varepsilon) \sum \limits_{k=m}^{n+1} \phi (k)} \\
&\leq \liminf_{n \to \infty} \frac{n (1+\varepsilon) \phi (n) +(n-1) - n \log n}{(1-\varepsilon) \sum \limits_{k=m}^{n} \phi (k)} \\
&= \liminf_{n \to \infty} \frac{(1+\varepsilon) \theta (n) + \left({(n-1)} \Big/{\sum \limits_{k=1}^{n} \phi (k)}\right) - \left({n \log n} \Big/{\sum \limits_{k=1}^{n} \phi (k)}\right)}{(1-\varepsilon) - (1-\varepsilon) \left({\sum \limits_{k=1}^{m-1} \phi (k)} \Big/{\sum \limits_{k=1}^{n} \phi (k)}\right)} \\
&= \frac{(1+\varepsilon) \theta + 0 - \gamma^{-1}}{(1-\varepsilon) - (1-\varepsilon)(0)} 
= \frac{(1+\varepsilon) \theta - \gamma^{-1}}{1-\varepsilon},
\end{align*}
where the second-to-last equality follows from \eqref{definition of theta} and Lemma \ref{n log n over sum lemma}. Combining the upper bounds for $\hdim E_\varepsilon (m, \phi)$ obtained in this and the preceding paragraph yields \eqref{sufficient condition}.

Since \eqref{sufficient condition} holds for any $m \geq K$, on combining \eqref{hdim E phi upper bound} and \eqref{sufficient condition} we infer that
\[
\hdim E(\phi) \leq \min \left\{ \frac{1+\varepsilon}{1-\varepsilon} \frac{1}{1+\xi}, \frac{(1+\varepsilon)\theta - \gamma^{-1}}{1-\varepsilon} \right\}.
\]
Notice that $\hdim E(\phi)$ is independent of $\varepsilon$. Since this inequality holds for arbitrary $\varepsilon \in (0,1)$, it follows by letting $\varepsilon \to 0^+$ that
\[
\hdim E(\phi) \leq \min \{ (1+\xi)^{-1}, \theta - \gamma^{-1} \}.
\]
\end{proof}

\begin{proof}[Proof of Theorem \ref{theorem E phi}]
If $\gamma \in [0, 1)$, then $1-\gamma^{-1} \in [-\infty,0)$, and hence $\hdim E(\phi) = 0 = \max \{ 0, (1-\gamma^{-1})(1+\xi)^{-1} \}$ by Lemma \ref{E phi lemma 1}. If $\gamma \in [1, \infty]$, by Lemmas \ref{E phi lower bound} and \ref{E phi upper bound} we have
\[
\frac{1-\gamma^{-1}}{1 + ( 1- \gamma^{-1} ) \xi} \leq \hdim E(\phi) \leq \min \{ (1+\xi)^{-1}, \theta - \gamma^{-1} \}.
\]
If, further, $\gamma \in [1, \infty)$, then $\xi=0$ by Lemma \ref{useful lemma}, and $\theta - \gamma^{-1} = 1 - \gamma^{-1}$ by Lemma \ref{n phi over sum}; thus, $\hdim E(\phi) = 1-\gamma^{-1} = \max \{ 0, (1-\gamma^{-1})(1+\xi)^{-1} \}$.  Otherwise, $1 - \gamma^{-1} = 1$ by convention, and $\theta - \gamma^{-1} = \theta \geq 1$ by convention and Lemma \ref{n phi over sum}; thus,
\[
\hdim E(\phi) = (1+\xi)^{-1} = \max \{ 0, (1-\gamma^{-1})(1+\xi)^{-1} \}.
\]
On combining these results, the theorem follows.
\end{proof}

\subsubsection{Proofs of corollaries to Theorem \ref{theorem E phi}} \label{corollary to theorem E phi}

\begin{proof} [Proof of Corollary \ref{corollary A alpha}]
The result for $\alpha \in [-\infty, 1)$ follows immediately from the fact that the digits are positive integers. In fact, for any $x \in \mathbb{I}$, since $d_n(x) \geq 1$ for all $n \in \mathbb{N}$, we have $\liminf_{n \to \infty} (d_n(x))^{1/n} \geq 1$. Hence, if $\lim_{n \to \infty} (d_n(x))^{1/n} = \alpha$, then $\alpha \geq 1$. Thus $A(\alpha) = \varnothing$ and $\hdim A(\alpha) = 0$ for any $\alpha \in [-\infty, 1)$.

We prove the corollary for $\alpha \in [1,\infty]$ by considering three separate cases: $\alpha = 1$, $\alpha \in (1, \infty)$, or $\alpha = \infty$. The proof makes use of Theorem \ref{theorem E phi}, and we choose a suitable function $\phi \colon \mathbb{N} \to (0,\infty)$ for each case.

{\sc Case I.}
Let $\alpha = 1$. Define $\phi (n) \coloneqq \sqrt{n}$ for each $n \in \mathbb{N}$. Clearly, $\phi \colon \mathbb{N} \to (0,\infty)$ is non-decreasing, and $\gamma \coloneqq \lim_{n \to \infty} \phi (n)/\log n = \infty$. For any $x \in E( \phi )$,
\[
\lim_{n \to \infty} \frac{\log d_n(x)}{n} = \lim_{n \to \infty} \frac{\log d_n(x)}{\phi (n)} \cdot \frac{\sqrt{n}}{n} = 1 \cdot 0 = 0,
\]
or equivalently $\lim_{n \to \infty} (d_n(x))^{1/n} = 1$. Hence, $E(\phi) \subseteq A(1)$. Note that since the map $x \mapsto \sqrt{x}$ is increasing on $[0, \infty)$, we have
\[
0 \leq \xi \coloneqq \limsup_{n \to \infty} \frac{\phi(n+1)}{\sum \limits_{k=1}^n \phi (k)} = \limsup_{n \to \infty} \frac{\sqrt{n+1}}{\sum \limits_{k=1}^n \sqrt{k}} \leq \lim_{n \to \infty} \frac{\sqrt{n+1}}{\displaystyle \int_0^n \sqrt{x} \, dx} = \lim_{n \to \infty} \frac{\sqrt{n+1}}{\dfrac{2}{3} n \sqrt{n}} = 0,
\]
and so $\xi = 0$. By Theorem \ref{theorem E phi}, $\hdim E(\phi) = (1+\xi)^{-1} = 1$, and it follows that $\hdim A(1) = 1$ by monotonicity of the Hausdorff dimension.

{\sc Case II.}
Let $\alpha \in (1, \infty)$. Define $\phi (n) \coloneqq n \log \alpha$ for each $n \in \mathbb{N}$. Clearly, $\phi \colon \mathbb{N} \to (0,\infty)$ is non-decreasing, and $\gamma \coloneqq \lim_{n \to \infty} \phi (n)/\log n = \infty$. For any $x \in E( \phi )$,
\[
\lim_{n \to \infty} \frac{\log d_n(x)}{n} = \lim_{n \to \infty} \frac{\log d_n(x)}{\phi (n)} \cdot \frac{n \log \alpha}{n} = 1 \cdot \log \alpha = \log \alpha,
\]
or equivalently $\lim_{n \to \infty} (d_n(x))^{1/n} = \alpha$. Hence, $E(\phi) \subseteq A(\alpha)$. Note that
\[
\xi \coloneqq \limsup_{n \to \infty} \frac{\phi(n+1)}{\sum \limits_{k=1}^n \phi (k)} = \lim_{n \to \infty} \frac{(n+1) \log \alpha}{\sum \limits_{k=1}^n ( k \log \alpha )} = \lim_{n \to \infty} \frac{(n+1) \log \alpha}{\dfrac{n(n+1)}{2} \log \alpha} = 0.
\]
By Theorem \ref{theorem E phi}, $\hdim E(\phi) = (1+\xi)^{-1} = 1$, and so $\hdim A(\alpha) = 1$ by monotonicity of the Hausdorff dimension.

{\sc Case III.}
Let $\alpha = \infty$. Define $\phi (n) \coloneqq n^2$ for each $n \in \mathbb{N}$. Clearly, $\phi \colon \mathbb{N} \to (0,\infty)$ is non-decreasing, and $\gamma \coloneqq \lim_{n \to \infty} \phi (n)/\log n = \infty$. For any $x \in E( \phi )$,
\[
\lim_{n \to \infty} \frac{\log d_n(x)}{n} = \lim_{n \to \infty} \frac{\log d_n(x)}{\phi (n)} \cdot \frac{n^2}{n} = 1 \cdot \infty = \infty,
\]
or equivalently $\lim_{n \to \infty} (d_n(x))^{1/n} = \infty$. Hence, $E(\phi) \subseteq A(\infty)$. Note that
\[
\xi \coloneqq \limsup_{n \to \infty} \frac{\phi(n+1)}{\sum \limits_{k=1}^n \phi (k)} = \lim_{n \to \infty} \frac{(n+1)^2}{\sum \limits_{k=1}^n k^2} = \lim_{n \to \infty} \frac{(n+1)^2}{\dfrac{n(n+1)(2n+1)}{6}} = 0.
\]
By Theorem \ref{theorem E phi}, $\hdim E(\phi) = (1+\xi)^{-1} = 1$, and so $\hdim A(\infty) = 1$ by monotonicity of the Hausdorff dimension.
\end{proof}

\begin{proof} [Proof of Corollary \ref{corollary LLN exceptional set}]
Take $\alpha = 1$ in Corollary \ref{corollary A alpha}. Clearly, $A(1)$ is a subset of the set on which \eqref{law of large numbers} fails. This is because the set where \eqref{law of large numbers} holds is exactly $A(e)$, and $A(1)$ and $A(e)$ are disjoint. But $\hdim A(1) = 1$ by the same corollary, and the result follows by monotonicity of the Hausdorff dimension.
\end{proof}

\begin{proof} [Proof of Corollary \ref{corollary E phi}]
Since $\phi \colon \mathbb{N} \to (0, \infty)$ is non-decreasing, it is enough to verify that $\gamma \coloneqq \lim_{n \to \infty} \phi (n)/\log n = \infty$. Then the result will follow from Theorem \ref{theorem E phi}. To see this, let $B> 0$. Since $\phi (n+1) - \phi (n) \to \infty$ as $n \to \infty$, we can find an $M = M(B) \in \mathbb{N}$ such that $\phi (n+1) - \phi (n) > B$ for all $n \geq M$. Then, for any $n \geq M$, we have $\phi (n) \geq \phi (M) + (n-M)B$. Thus,
\[
\liminf_{n \to \infty} \frac{\phi (n)}{\log n} \geq \liminf_{n \to \infty} \frac{\phi (M) + (n-M) B}{\log n} = \infty,
\]
and this proves that $\gamma = \infty$, as was to be shown.
\end{proof}

\subsubsection{Proof of Theorem \ref{theorem A kappa}} \label{Proof of theorem A kappa}

It should be mentioned that the idea of the following proof is borrowed from \cite[proof of Corollary 1]{LW01} in which an analogous set for Engel expansions is discussed.

\begin{proof} [Proof of Theorem \ref{theorem A kappa}]
The result for the case where $\kappa \in (-\infty, 0]$ follows immediately from the fact that the digits are positive integers. In fact, for any $x \in \mathbb{I}$, since $d_n(x) \geq 1$ for all $n \in \mathbb{N}$, we have $\log d_n (x) \geq 0 \geq \kappa n$ for all $n \in \mathbb{N}$. Hence, $A_\kappa = \mathbb{I}$, and obviously $\hdim A_\kappa = 1$.

Let $\kappa \in (0, \infty)$. Choose $\alpha \in \mathbb{R}$ so that $\alpha > e^\kappa$, or equivalently $\log \alpha > \kappa$. Then $\alpha > 1$. Define $u_k \coloneqq 2 \alpha^k$ for each $k \in \mathbb{N}$. Clearly, $2 \leq u_k \leq u_{k+1}$ for all $k \in \mathbb{N}$. Then
\begin{align*}
\eta
&\coloneqq \limsup_{n \to \infty} \frac{n \log n + \log u_{n+1}}{\log \left( \prod \limits_{k=1}^n u_k \right)} 
= \lim_{n \to \infty} \frac{n \log n + \log (2\alpha^{n+1})}{\sum \limits_{k=1}^n (\log 2 + k \log \alpha)} 
= 0 .
\end{align*}
It follows from Lemma \ref{dimension of E star} that $\hdim \mathcal{E}^* = (1+\eta)^{-1} = 1$, where $\mathcal{E}^*$ is defined as in \eqref{definition of E un}, i.e.,
\[
\mathcal{E}^* = \left\{ x \in (0,1] : 2 n \alpha^n < d_n(x) \leq 2 (n+1) \alpha^n \text{ for all } n \in \mathbb{N} \right\}.
\]
Then, for any $x \in \mathcal{E}^*$,
\[
\kappa n < \log (2n) + \kappa n < \log (2n) + n \log \alpha < \log d_n(x) \quad \text{for all } n \in \mathbb{N}.
\]
Thus $\mathcal{E}^* \subseteq A_\kappa$, and so $\hdim A_\kappa = 1$ by monotonicity of the Hausdorff dimension.
\end{proof}

\subsection{The ratio between two consecutive digits}

This subsection is devoted to proving the third main result of this paper and its corollary.

\begin{proof} [Proof of Theorem \ref{theorem B alpha}]
The result for $\alpha \in [-\infty, 1)$ follows immediately from the strict increasing condition for the digits (Proposition \ref{digit condition lemma}(\ref{digit condition lemma 1})). In fact, for any $x \in \mathbb{I}$, since $d_{n+1} (x) > d_n(x)$ for all $n \in \mathbb{N}$, we see that $\liminf_{n \to \infty} d_{n+1}(x)/d_n(x) \geq 1$. Hence, if $\lim_{n \to \infty} d_{n+1}(x)/d_n(x) = \alpha$, then $\alpha \geq 1$. Thus, $B(\alpha) = \varnothing$ and $\hdim B(\alpha) = 0$ for any $\alpha \in [-\infty, 1)$.

For $\alpha \in [1, \infty]$, the proof is mainly based on Lemma \ref{dimension of E star}. We choose suitable sequences $( u_k )_{k \in \mathbb{N}}$ to prove the result. There are three cases to consider according as $\alpha = 1$, $\alpha \in (1, \infty)$, or $\alpha = \infty$.

{\sc Case I}.
Let $\alpha=1$. Define $u_k \coloneqq e^{\sqrt{k}}$ for each $k \in \mathbb{N}$. Clearly, $2 \leq u_k \leq u_{k+1}$ for all $k \in \mathbb{N}$. Since the map $x \mapsto \sqrt{x}$ is increasing on $[0, \infty)$, we obtain
\begin{align*}
0 \leq \eta
&\coloneqq \limsup_{n \to \infty} \frac{n \log n + \log u_{n+1}}{\log \left( \prod \limits_{k=1}^n u_k \right)} 
= \limsup_{n \to \infty} \frac{n \log n + \sqrt{n+1}}{\sum \limits_{k=1}^n \sqrt{k}} \\
&\leq \lim_{n \to \infty} \frac{n \log n + \sqrt{n+1}}{\displaystyle \int_0^n \sqrt{x} \, dx} 
= \lim_{n \to \infty} \frac{n \log n + \sqrt{n+1}}{\dfrac{2}{3}n \sqrt{n}} = 0,
\end{align*}
and hence $\eta = 0$. It follows from Lemma \ref{dimension of E star} that $\hdim \mathcal{E}^* = (1+\eta)^{-1} = 1$, where $\mathcal{E}^*$ is defined as in \eqref{definition of E un}, i.e.,
\[
\mathcal{E}^* = \left\{ x \in (0,1] : n e^{\sqrt{n}} < d_n(x) \leq (n+1) e^{\sqrt{n}} \text{ for all } n \in \mathbb{N} \right\}.
\]
Then, for any $x \in \mathcal{E}^*$,
\[
e^{\sqrt{n+1}-\sqrt{n}} < \frac{d_{n+1}(x)}{d_n(x)} < \left( 1 + \frac{2}{n} \right) e^{\sqrt{n+1}-\sqrt{n}} \quad \text{for all } n \in \mathbb{N},
\]
and hence $\lim_{n \to \infty} d_{n+1}(x)/d_n(x) = 1$. Thus $\mathcal{E}^* \subseteq B(1)$, and we conclude that $\hdim B(1) = 1$ by monotonicity of the Hausdorff dimension.

{\sc Case II}.
Let $\alpha \in (1, \infty)$. Define $u_k \coloneqq e^{\log 2 + k \log \alpha} = 2 \alpha^k$ for each $k \in \mathbb{N}$. Clearly, $2 \leq u_k \leq u_{k+1}$ for all $k \in \mathbb{N}$. Then, as in the proof of Theorem \ref{theorem A kappa} (see Section \ref{Proof of theorem A kappa}), we find that $\hdim \mathcal{E}^* = 1$, where $\mathcal{E}^*$ is defined as in \eqref{definition of E un}, i.e.,
\[
\mathcal{E}^* = \left\{ x \in (0,1] : 2 n \alpha^n < d_n(x) \leq 2 (n+1) \alpha^n \text{ for all } n \in \mathbb{N} \right\}.
\]
Then, for any $x \in \mathcal{E}^*$, 
\[
\alpha < \frac{d_{n+1}(x)}{d_n(x)} < \left( 1 + \frac{2}{n} \right) \alpha \quad \text{for all } n \in \mathbb{N},
\]
and hence $\lim_{n \to \infty} d_{n+1}(x)/d_n(x) = \alpha$. Therefore, $\mathcal{E}^* \subseteq B(\alpha)$, and we conclude that $\hdim B(\alpha) = 1$ by monotonicity of the Hausdorff dimension.

{\sc Case III}.
Let $\alpha = \infty$. Define $u_k \coloneqq e^{k^2}$ for each $k \in \mathbb{N}$. Clearly, $2 \leq u_k \leq u_{k+1}$ for all $k \in \mathbb{N}$. Then
\begin{align*}
\eta
&\coloneqq \limsup_{n \to \infty} \frac{n \log n + \log u_{n+1}}{\log \left( \prod \limits_{k=1}^n u_k \right)} 
= \lim_{n \to \infty} \frac{n \log n + (n+1)^2}{\sum \limits_{k=1}^n k^2} 
= \lim_{n \to \infty} \frac{n \log n + (n+1)^2}{\dfrac{n(n+1)(2n+1)}{6}} 
= 0.
\end{align*}
It follows from Lemma \ref{dimension of E star} that $\hdim \mathcal{E}^* = (1+\eta)^{-1} = 1$, where $\mathcal{E}^*$ is defined as in \eqref{definition of E un}, i.e.,
\[
\mathcal{E}^* = \left\{ x \in (0,1] : n e^{n^2} < d_n(x) \leq (n+1) e^{n^2} \text{ for all } n \in \mathbb{N} \right\}.
\]
Then, for any $x \in \mathcal{E}^*$, 
\[
e^{2n+1} < \frac{d_{n+1}(x)}{d_n(x)}  \quad \text{for all } n \in \mathbb{N},
\]
and hence $\lim_{n \to \infty} d_{n+1}(x)/d_n(x) = \infty$. Therefore, $\mathcal{E}^* \subseteq B(\infty)$, and we conclude that $\hdim B(\infty) = 1$ by monotonicity of the Hausdorff dimension.
\end{proof}

\begin{remarks} \phantomsection \label{remark to the proof of theorem B alpha}
\begin{enumerate}[label=\upshape(\arabic*), ref=\arabic*, leftmargin=*, widest=2]
\item \label{remark to the proof of theorem B alpha 1}
We claim that $B(\alpha) \subseteq A(\alpha)$ for each $\alpha \in [1, \infty]$, where $A(\alpha)$ is defined as in \eqref{definition of A alpha}. To see this, let $\alpha \in [1, \infty)$. Take $x \in B(\alpha)$. Then, for any $\varepsilon \in (0, \alpha)$, there exists $K = K(\varepsilon) \in \mathbb{N}$ such that $(\alpha-\varepsilon) d_n(x) < d_{n+1}(x) < (\alpha+\varepsilon) d_n(x)$ for all $n \geq K$. Hence, for any $n>K$, we have $(\alpha-\varepsilon)^{n-K} d_K(x) < d_n(x) < (\alpha+\varepsilon)^{n-K} d_K(x)$. It follows that
\[
\alpha-\varepsilon \leq \liminf_{n \to \infty} (d_n(x))^{1/n} \leq \limsup_{n \to \infty} (d_n(x))^{1/n} \leq \alpha+\varepsilon,
\]
and by letting $\varepsilon \to 0^+$, we conclude that $(d_n(x))^{1/n} \to \alpha$ as $n \to \infty$, i.e., $x \in A(\alpha)$. This proves $B(\alpha) \subseteq A(\alpha)$ for each $\alpha \in [1, \infty)$. The inclusion $B(\infty) \subseteq A(\infty)$ can be proved in a similar fashion. Since $\hdim B(\alpha) = 1$ for each $\alpha \in [1, \infty]$, it follows by monotonicity of the Hausdorff dimension that $\hdim A(\alpha)=1$ for each $\alpha \in [1, \infty]$. This is an alternative proof of Corollary \ref{corollary A alpha} promised in Remarks \ref{remark to theorem B alpha}(\ref{remark to theorem B alpha 1}).

\item \label{remark to the proof of theorem B alpha 2}
We illustrated above that Corollary \ref{corollary A alpha} is a consequence of Theorem \ref{theorem B alpha}. Recall that we used Theorem \ref{theorem E phi} in the original proof of Corollary \ref{corollary A alpha} (see Section \ref{corollary to theorem E phi}).
However, Theorem \ref{theorem B alpha} does not seem to be a direct consequence of Theorem \ref{theorem E phi}. For instance, $d_{n+1}(x)/d_n(x) \to e$ as $n \to \infty$ for any $x \in B(e)$, and this suggests that $d_n(x)$ behaves like $e^n$ (in some sense) for all sufficiently large $n$. So, if one wishes to use Theorem \ref{theorem E phi}, it is natural to consider the function $\phi \colon \mathbb{N} \to (0, \infty)$ defined by $\phi (n) \coloneqq n$ for each $n \in \mathbb{N}$. In this case, $\hdim E(\phi) = 1$ by Theorem \ref{theorem E phi} (see {\sc Case II} of the proof of Corollary \ref{corollary A alpha} for the details). Let $\sigma \coloneqq (\sigma_n)_{n \in \mathbb{N}}$ be a sequence of positive integers given by
\[
\sigma_n \coloneqq \begin{cases}
\lfloor e^{n-1+\sqrt{n}} \rfloor, &\text{if $n$ is odd}; \\
\lfloor e^{n+ \sqrt{n}} \rfloor, &\text{if $n$ is even}.
\end{cases}
\]
We claim that $\sigma \in \Sigma_{\admissible}$. To see this, first note that $g_1 (x) \coloneqq e^{\sqrt{x+1}} - e^{\sqrt{x}}$ defined on $(1, \infty)$ is increasing. This is because $g_1'(x) = (1/2) (h(x+1)-h(x))$ where, $h(x) \coloneqq e^{\sqrt{x}} / \sqrt{x}$ and $h'(x) = [e^{\sqrt{x}}(\sqrt{x}-1)]/(2 x \sqrt{x}) > 0$ on $(1, \infty)$, so that $h(x+1)>h(x)$ and $g_1'(x) > 0$ on $(1, \infty)$. Hence, for any $k \in \mathbb{N}_0$,
\begin{align*}
\sigma_{2k+3} - \sigma_{2k+2}
&\geq (e^{2k+2+\sqrt{2k+3}}-1) - e^{2k+2+\sqrt{2k+2}} = e^{2k+2} (e^{\sqrt{2k+3}} - e^{\sqrt{2k+2}}) - 1\\
&\geq e^2 (e^{\sqrt{3}}-e^{\sqrt{2}}) - 1 > 10,
\end{align*}
where the last inequality follows from direct calculation. Similarly, the map $g_2(x) \coloneqq e^{2+\sqrt{x+1}} - e^{\sqrt{x}}$ is increasing on $(1, \infty)$, since $g_2'(x) = (1/2) (e^2 h(x+1) - h(x)) > g_1'(x) > 0$ on $(1, \infty)$. Then, for any $k \in \mathbb{N}_0$,
\begin{align*}
\sigma_{2k+2}-\sigma_{2k+1} 
&\geq (e^{2k+2+\sqrt{2k+2}}-1) - e^{2k+\sqrt{2k+1}} = e^{2k} (e^{2+\sqrt{2k+2}} - e^{\sqrt{2k+1}})-1 \\
&\geq e^0 (e^{2+\sqrt{2}}-e^1)-1 > 26,
\end{align*}
where the last inequality follows from direct calculation. Thus, $\sigma \in \Sigma_\infty$ by definition, and by Proposition \ref{strict increasing condition}, $\sigma \in \Sigma_{\admissible}$, i.e., there exists $x \in (0,1]$ such that $(d_n(x))_{n \in \mathbb{N}} = \sigma$. But then $(\log d_n(x)) / \phi (n) = (\log d_n(x)) / n \to 1$ as $n \to \infty$ with
\[
\lim_{k \to \infty} \frac{d_{2k+1}(x)}{d_{2k}(x)} = \lim_{k \to \infty} e^{\sqrt{2k+1}-\sqrt{2k}} = 1
\]
and
\[
\lim_{k \to \infty} \frac{d_{2k+2}(x)}{d_{2k+1}(x)} = \lim_{k \to \infty} e^{2 + \sqrt{2k+2} - \sqrt{2k+1}} = e^2.
\]
Hence, $x \in E(\phi)$ but $x \not \in B(e)$, and thus $E(\phi) \not \subseteq B(e)$, in which case we cannot use the monotonicity argument to deduce that $\hdim B(e) = 1$.
\end{enumerate}
\end{remarks}

\begin{proof} [Proof of Corollary \ref{corollary B kappa}]
Recall that for any $x \in \mathbb{I}$, we have $d_{n+1}(x) > d_n(x)$ for all $n \in \mathbb{N}$ by Proposition \ref{digit condition lemma}(\ref{digit condition lemma 1}). It is clear that if $\kappa \leq 1$, then $B_\kappa = \varnothing$, and hence $\hdim B_\kappa = 0$.

Let $\kappa > 1$. Fix $\alpha \in [1, \kappa)$. Let $\varepsilon > 0$ be small enough so that $\alpha+\varepsilon < \kappa$. If $x$ is any element in
\[
B(\alpha) = \left\{ x \in (0,1] : \lim_{n \to \infty} \frac{d_{n+1}(x)}{d_n(x)} = \alpha \right\},
\]
we can find a positive integer $N = N(x, \varepsilon)$ such that $-\varepsilon < d_{n+1}(x)/d_n(x) - \alpha \leq \varepsilon$, or equivalently
\[
(\alpha-\varepsilon) d_n(x) < d_{n+1}(x) \leq (\alpha+\varepsilon) d_n(x)
\]
for all $n \geq N$. For each $m \in \mathbb{N}$, we define
\[
B_\varepsilon (m, \alpha) \coloneqq \{ x \in (0,1]: (\alpha-\varepsilon) d_n(x) < d_{n+1}(x) \leq (\alpha+\varepsilon) d_n(x) \text{ for all } n \geq m \}.
\]

Now, let
\[
B_\varepsilon (\alpha) \coloneqq \bigcup_{m \in \mathbb{N}} B_\varepsilon (m, \alpha).
\]
Clearly, $B(\alpha) \subseteq B_\varepsilon (\alpha)$. Observe that since $\alpha \geq 1$, we have $\hdim B(\alpha) = 1$ by Theorem \ref{theorem B alpha}. Using monotonicity and countable stability of the Hausdorff dimension, we deduce that
\[
1 = \hdim B(\alpha) \leq \hdim B_\varepsilon (\alpha) = \sup_{m \in \mathbb{N}} \hdim B_\varepsilon (m, \alpha) = \hdim B_\varepsilon (1,\alpha),
\]
where the last equality follows from Lemma \ref{preservation of dimension}(\ref{preservation of dimension 2}) (by taking $K \coloneqq 1$, $h_1 (n) \coloneqq (\alpha-\varepsilon) n$, $h_2(n) \coloneqq n$, $h_3(n) \coloneqq (\alpha+\varepsilon) n$, and $\mathcal{S}(m, h_1, h_2, h_3) \coloneqq B_\varepsilon (m, \alpha)$).
It follows that $\hdim B_\varepsilon (1, \alpha) = 1$. Notice that for any $x \in B_\varepsilon (1, \alpha)$,
\[
\frac{d_{n+1}(x)}{d_n(x)} \leq \alpha + \varepsilon < \kappa \quad \text{for all } n \in \mathbb{N},
\]
so that $x \in B_\kappa$. Therefore, $B_\varepsilon (1, \alpha) \subseteq B_\kappa$, and we conclude that $\hdim B_\kappa = 1$ by monotonicity of the Hausdorff dimension.
\end{proof}

\begin{remark} \label{remark to proof of corollary B kappa}
We show that each $B_\kappa$, $\kappa \in (1, e)$, is an exceptional set to the LLN. To see this, let $\kappa \in (1, e)$, and suppose $x$ is in $B_\kappa$. Then for any $n \in \mathbb{N}$,
\[
d_{n+1}(x) = d_1(x) \frac{d_2(x)}{d_1(x)} \dotsm \frac{d_{n+1}(x)}{d_n(x)} \leq d_1(x) \kappa^n.
\]
Hence, $\limsup_{n \to \infty} (d_n(x))^{1/n} \leq \kappa < e$, and so $(d_n(x))^{1/n} \not \to e$ as $n \to \infty$. But $\hdim B_\kappa = 1$ by Corollary \ref{corollary B kappa}, and thus monotoncity of the Hausdorff dimension implies Corollary \ref{corollary LLN exceptional set}, which answers the Pierce expansion version of \ref{Q1}. Note that in particular, $B_2$ is an exceptional set to the LLN, as mentioned in Remarks \ref{remark B k}(\ref{remark B k 2}).
\end{remark}

\subsection{The ratio between the logarithms of two consecutive digits}

This subsection is devoted to proving the fourth main result of this paper. 

The result for the case where $\alpha = 1$ is immediate from the LLN.

\begin{lemma} \label{alpha one case}
The set $F(1)$ is of full Lebesgue measure on $(0,1]$, and consequently $\hdim F(1) = 1$.
\end{lemma}

\begin{proof}
Let $x \in A(e)$, where $A(e)$ is defined as in \eqref{definition of A alpha}. Then $(d_n(x))^{1/n} \to e$, or equivalently $(\log d_n(x))/n \to 1$, as $n \to \infty$, and we infer that $x \in F(1)$ by observing that
\[
\lim_{n \to \infty} \frac{\log d_{n+1}(x)}{\log d_n(x)} = \lim_{n \to \infty} \frac{\log d_{n+1}(x)}{n+1} \frac{n}{\log d_n(x)} \frac{n+1}{n} = 1 \cdot 1 \cdot 1 = 1.
\]
Hence, $A(e) \subseteq F(1)$. By the LLN, we know that $A(e)$ is of full Lebesgue measure on $(0,1]$. Thus, $F(1)$ has full Lebesgue measure on $(0,1]$, and so $\hdim F(1)=1$.
\end{proof}

Next, we obtain the lower bound for $\hdim F(\alpha)$, where $\alpha > 1$, by using Theorem \ref{theorem E phi}.

\begin{lemma} \label{F alpha lower bound}
$\hdim F(\alpha) \geq \alpha^{-1}$ for each $\alpha \in (1, \infty)$.
\end{lemma}

\begin{proof}
Fix $\alpha \in (1, \infty)$ and let $\phi (n) \coloneqq \alpha^n$ for each $n \in \mathbb{N}$. Clearly, $\phi \colon \mathbb{N} \to (0,\infty)$ is non-decreasing, and $\gamma \coloneqq \lim_{n \to \infty} \phi (n)/\log n = \infty$. Then, for any $x \in E(\phi)$, where $E(\phi)$ is defined as in \eqref{definition of E phi}, we have
\[
\lim_{n \to \infty} \frac{\log d_{n+1}(x)}{\log d_n(x)} = \lim_{n \to \infty} \frac{\log d_{n+1}(x)}{\phi (n+1)} \frac{\phi (n)}{\log d_n(x)} \frac{\alpha^{n+1}}{\alpha^n} = 1 \cdot 1 \cdot \alpha = \alpha,
\]
and hence $E(\phi) \subseteq F(\alpha)$. Note that
\[
\xi \coloneqq \limsup_{n \to \infty} \frac{\phi(n+1)}{\sum \limits_{k=1}^n \phi (k)} = \lim_{n \to \infty} \frac{\alpha^{n+1}}{\sum \limits_{k=1}^n \alpha^k} = \lim_{n \to \infty} \frac{\alpha^{n+1}}{\dfrac{\alpha (\alpha^n-1)}{\alpha-1}} = \alpha - 1.
\]
By Theorem \ref{theorem E phi}, we have $\hdim E(\phi) = (1+\xi)^{-1} = \alpha^{-1}$, and we conclude that $\hdim F(\alpha) \geq \alpha^{-1}$ by monotonicity of the Hausdorff dimension.
\end{proof}

We now determine the Hausdorff dimension of $F(\infty)$. A similar proof will be used later to establish the desired upper bound for $\hdim F(\alpha)$, where $\alpha \in (1, \infty)$.

\begin{lemma} \label{F infinity}
$\hdim F( \infty) = 0$.
\end{lemma}

\begin{proof}
Let $B \in (1, \infty)$. If $x \in F(\infty)$, then, since $\log d_{n+1}(x) / \log d_n(x) \to \infty$ as $n \to \infty$, we can find a positive integer $N = N(x, B) \geq 2$ such that
\[
\frac{\log d_{n+1}(x)}{\log d_n(x)} > B, \quad \text{or equivalently} \quad d_n^{B}(x) < d_{n+1}(x)
\]
for all $n \geq N$. For each positive integer $m \geq 2$, we define
\[
F_B (m, \infty) \coloneqq \{ x \in (0,1] : d_n^{B}(x) < d_{n+1}(x) \text{ for all } n \geq m \}.
\]

Now, let
\begin{align*}
F_B (\infty)
&\coloneqq \bigcup_{m \geq 2} F_B (m, \infty).
\end{align*}
Clearly, $F(\infty) \subseteq F_B (\infty)$. Using monotonicity and countable stability of the Hausdorff dimension, we deduce that
\begin{align} \label{F two infinity 1}
\hdim F(\infty) \leq \hdim F_B (\infty) = \sup_{m \geq 2} \hdim F_B (m, \infty) = \hdim F_B (2, \infty),
\end{align}
where the last equality follows from Lemma \ref{preservation of dimension}(\ref{preservation of dimension 1}) (by taking $K \coloneqq 2$, $h_1 (n) \coloneqq n^B$, $h_2(n) \coloneqq n$, and $\mathcal{S}(m, h_1, h_2) \coloneqq F_B (m, \infty)$).

For each $(\tau_1, \tau_2) \in \Sigma_2$, let
\begin{align*}
F_B^{(\tau_1, \tau_2)}(2, \infty) \coloneqq \{ x \in (0,1] : d_1(x) = \tau_1, \, d_2(x) &= \tau_2, \\
\text{ and } d_n^{B}(x) &< d_{n+1}(x) \text{ for all } n \geq 2 \}.
\end{align*}
Then
\[
F_B (2, \infty) = \bigcup_{(\tau_1, \tau_2) \in \Sigma_2} F_B^{(\tau_1, \tau_2)}(2, \infty).
\]
Since $\Sigma_2 \subseteq \mathbb{N}^2$ is countable, again by countable stability we have
\begin{align} \label{F two infinity 2}
\hdim F_B (2, \infty) = \sup_{(\tau_1, \tau_2) \in \Sigma_2} \hdim F_B^{(\tau_1, \tau_2)}(2, \infty).
\end{align}

Fix $(\tau_1, \tau_2) \in \Sigma_2$. We shall make use of a symbolic space. Let $\widehat{\Upsilon}_2 \coloneqq \{ (\tau_1, \tau_2) \}$. For any $n \geq 3$, let
\begin{align*}
\widehat{\Upsilon}_n \coloneqq \{ (\sigma_k)_{k=1}^n \in \Sigma_n : \sigma_1 = \tau_1, \, \sigma_2 = \tau_2, \text{ and } \sigma_k^{B} < \sigma_{k+1} \text{ for all } 2 \leq k \leq n-1 \}.
\end{align*}
It is clear that $\widehat{\Upsilon}_n \neq \varnothing$ since $\mathbb{N}$ is not bounded above. Define
\[
\widehat{\Upsilon} \coloneqq \bigcup_{n \geq 2} \widehat{\Upsilon}_n.
\]
For any $\sigma \coloneqq (\sigma_k)_{k=1}^n \in \widehat{\Upsilon}_n$, $n \geq 2$, we define the {\em $n$th level basic interval} $\widehat{\mathcal{J}} (\sigma)$ by
\[
\widehat{\mathcal{J}} (\sigma) \coloneqq \bigcup_{j \geq \lfloor \sigma_n^{B} \rfloor +1} 
\overline{\mathcal{I} (\sigma^{(j)})}.
\]

We bound the Hausdorff dimension of $F_B^{(\tau_1, \tau_2)} (2, \infty)$ from above by considering its covering. Observe that, similar to \eqref{equivalent definition of E ln rn},
\begin{align*}
F_B^{(\tau_1, \tau_2)} (2, \infty) 
= \bigcap_{n \geq 2} \bigcup_{\sigma \in \widehat{\Upsilon}_n} \widehat{\mathcal{J}} (\sigma).
\end{align*}
Then for any $n \geq 2$, the collection $\{ \widehat{\mathcal{J}}(\sigma) : \sigma \in \widehat{\Upsilon}_n \}$ of all $n$th level basic intervals is a covering of $F_B^{(\tau_1, \tau_2)} (2, \infty)$. By using a general inequality $\lfloor x \rfloor + 1 > x$ for any $x \in \mathbb{R}$ and \eqref{diam I sigma}, we find that
\begin{align} \label{length of J sigma hat}
\begin{aligned}
\diam \widehat{\mathcal{J}} (\sigma)
&= \len \widehat{\mathcal{J}} (\sigma) \\
&= \sum_{j \geq \lfloor \sigma_n^{B} \rfloor+1} \left[ \left( \prod_{k=1}^n \frac{1}{\sigma_k} \right) \left( \frac{1}{j} - \frac{1}{j+1} \right) \right]
= \left( \prod_{k=1}^n \frac{1}{\sigma_k} \right) \frac{1}{\lfloor \sigma_n^{B} \rfloor + 1} \\
&\leq \left( \prod_{k=1}^n \frac{1}{\sigma_k} \right) \frac{1}{\sigma_n^{B}}
= \left( \prod_{k=1}^{n-1} \frac{1}{\sigma_k} \right) \frac{1}{\sigma_n^{1+B}}.
\end{aligned}
\end{align}

Fix $s \in (B^{-1}, \infty)$. We will show that the $s$-dimensional Hausdorff measure of $F_B^{(\tau_1, \tau_2)} (2, \infty)$ is finite. 
Put $A \coloneqq \sum_{n \in \mathbb{N}} n^{-sB}$, which is finite since $sB>1$. For each $n \geq 2$, by making use of the implication
\[
j \in \mathbb{N} \cap \big[ \lfloor \sigma_{n-1}^{B} \rfloor + 1, \infty \big) \subseteq \big( \sigma_{n-1}^{B}, \infty \big)
\implies
\frac{\sigma_{n-1}^{B}}{j} < 1,
\]
we observe that 
\begin{align} \label{less than A}
\sum_{j \geq \lfloor \sigma_{n-1}^{B} \rfloor + 1} \left( \frac{\sigma_{n-1}^{B}}{j^{1+B}} \right)^{s} 
= \sum_{j \geq \lfloor \sigma_{n-1}^{B} \rfloor + 1} \left[ \left( \frac{\sigma_{n-1}^{B}}{j} \right)^{s} \left( \frac{1}{j^{B}} \right)^s \right] 
\leq \sum_{j \geq \lfloor \sigma_{n-1}^{B} \rfloor + 1} \frac{1}{j^{sB}} 
\leq A.
\end{align}
Take an integer $M \geq 2$ such that $\sum_{j \geq n} j^{-sB} < 1$ for all $n > M$. By Proposition \ref{strict increasing lemma}(\ref{strict increasing lemma 1}), we have $\lfloor \sigma_{n-1}^B \rfloor \geq \lfloor (n-1)^B \rfloor \geq \lfloor n-1 \rfloor = n-1$, where the second inequality holds since $B>1$, and so
\begin{align} \label{less than one}
\sum_{j \geq \lfloor \sigma_{n-1}^{B} \rfloor + 1} \left( \frac{\sigma_{n-1}^{B}}{j^{1+B}} \right)^{s} \leq \sum_{j \geq \lfloor \sigma_{n-1}^{B} \rfloor + 1} \frac{1}{j^{sB}} \leq \sum_{j \geq n} \frac{1}{j^{sB}} < 1
\end{align}
for every $n>M$. Hence, for any $n>M$ large enough, by using \eqref{length of J sigma hat}, \eqref{less than one}, and \eqref{less than A}, we obtain
\begin{align*} 
\sum_{\sigma \in \widehat{\Upsilon}_n} [\diam \widehat{\mathcal{J}} (\sigma)]^s
&\leq \sum_{(\sigma_k)_{k=1}^n \in \widehat{\Upsilon}_n} \left[ \left( \prod_{k=1}^{n-1} \frac{1}{\sigma_k} \right) \frac{1}{\sigma_n^{1+B}} \right]^{s} \\
&= \sum_{(\sigma_k)_{k=1}^{n-1} \in \widehat{\Upsilon}_{n-1}} \left[  \left[ \left( \prod_{k=1}^{n-2} \frac{1}{\sigma_k} \right) \frac{1}{\sigma_{n-1}^{1+B}} \right]^{s} \sum_{j \geq \lfloor \sigma_{n-1}^{B} \rfloor + 1} \left( \frac{\sigma_{n-1}^{B}}{j^{1+B}} \right)^{s} \right]  \\
&\leq \sum_{(\sigma_k)_{k=1}^{n-2} \in \widehat{\Upsilon}_{n-2}}  \left[ \left[ \left( \prod_{k=1}^{n-3} \frac{1}{\sigma_k} \right) \frac{1}{\sigma_{n-2}^{1+B}} \right]^{s} \sum_{j \geq \lfloor \sigma_{n-2}^{B} \rfloor + 1} \left( \frac{\sigma_{n-2}^{B}}{j^{1+B}} \right)^{s} \right]  \\
&\leq \dotsb \leq \sum_{(\sigma_k)_{k=1}^{M-1} \in \widehat{\Upsilon}_{M-1}} \left[ \left[ \left( \prod_{k=1}^{M-2} \frac{1}{\sigma_k} \right) \frac{1}{\sigma_{M-1}^{1+B}} \right]^{s} \sum_{j \geq \lfloor \sigma_{M-1}^{B} \rfloor + 1} \left( \frac{\sigma_{M-1}^{B}}{j^{1+B}} \right)^{s} \right]  \\
&\leq A \sum_{(\sigma_k)_{k=1}^{M-2} \in \widehat{\Upsilon}_{M-2}} \left[ \left[ \left( \prod_{k=1}^{M-3} \frac{1}{\sigma_k} \right) \frac{1}{\sigma_{M-2}^{1+B}} \right]^{s} \sum_{j \geq \lfloor \sigma_{M-2}^B \rfloor + 1} \left( \frac{\sigma_{M-2}^B}{j^{1+B}} \right)^s \right]  \\
&\leq \dotsb
\leq A^{M-3} \sum_{(\sigma_k)_{k=1}^{2} \in \widehat{\Upsilon}_{2}} \left[ \left( \frac{1}{\sigma_1 \sigma_{2}^{1+B}} \right)^{s} \sum_{j \geq \lfloor \sigma_{2}^B \rfloor + 1} \left( \frac{\sigma_{2}^B}{j^{1+B}} \right)^s \right]  \\
&\leq A^{M-2} \sum_{(\sigma_1, \sigma_2) \in \widehat{\Upsilon}_2} \left( \frac{1}{\sigma_1 \sigma_2^{1+B}} \right)^s 
= \frac{A^{M-2}}{\tau_1^s \tau_2^{s(1+B)}}.
\end{align*}
It follows that
\begin{align*}
\mathcal{H}^{s} (F_B^{(\tau_1, \tau_2)} (2, \infty))
&\leq \liminf_{n \to \infty} \sum_{\sigma \in \widehat{\Upsilon}_n} [\diam \widehat{\mathcal{J}} (\sigma) ]^{s}
< \infty,
\end{align*}
as was to be shown.
Since $s \in (B^{-1}, \infty)$ was arbitrary, we deduce that $\hdim F_B^{(\tau_1, \tau_2)} (2, \infty) \leq B^{-1}$. The sequence $(\tau_1, \tau_2) \in \Sigma_2$ was arbitrary as well, so that on combining this result with \eqref{F two infinity 1} and \eqref{F two infinity 2}, we infer that 
\begin{align} \label{F infinity F B 2 infinity}
\hdim F(\infty) \leq \hdim F_B(2,\infty) \leq B^{-1}. 
\end{align}
Notice that $F(\infty)$ is independent of $B$, and so is $\hdim F(\infty)$. Therefore, by letting $B \to \infty$, we conclude that $\hdim F (\infty) = 0$.
\end{proof}

As promised, using the preceding proof and a similar argument, we establish the upper bound for $\hdim F(\alpha)$, where $\alpha \in (1, \infty)$.

\begin{lemma} \label{F alpha upper bound}
$\hdim F(\alpha) \leq \alpha^{-1}$ for each $\alpha \in (1, \infty)$.
\end{lemma}

\begin{proof}
Fix $\alpha \in (1, \infty)$. Let $\varepsilon > 0$ be small enough so that $\alpha-\varepsilon>1$. If $x \in F(\alpha)$, then since $\log d_{n+1}(x) / \log d_n(x) \to \alpha$ as $n \to \infty$, we can find a positive integer $N = N(x, B) \geq 2$ such that
\[
-\varepsilon < \frac{\log d_{n+1}(x)}{\log d_n(x)} - \alpha \leq \varepsilon, \quad \text{or equivalently} \quad d_n^{\alpha-\varepsilon}(x) < d_{n+1}(x) \leq d_n^{\alpha+\varepsilon}(x),
\]
for all $n \geq N$. For each positive integer $m \geq 2$, we define
\[
\widehat{F}_\varepsilon (m, \alpha) \coloneqq \{ x \in (0,1] : d_n^{\alpha-\varepsilon}(x) < d_{n+1}(x) \leq d_n^{\alpha+\varepsilon}(x) \text{ for all } n \geq m \}.
\]

Now, let
\begin{align*}
\widehat{F}_\varepsilon (\alpha)
&\coloneqq \bigcup_{m \geq 2} \widehat{F}_\varepsilon (m, \alpha).
\end{align*}
Clearly, $F (\alpha) \subseteq \widehat{F}_\varepsilon (\alpha)$. Using monotonicity and countable stability of the Hausdorff dimension, we deduce that
\begin{align} \label{F two alpha 1}
\hdim F (\alpha) \leq \hdim \widehat{F}_\varepsilon (\alpha) = \sup_{m \geq 2} \hdim \widehat{F}_\varepsilon (m, \alpha) = \hdim \widehat{F}_\varepsilon (2, \alpha),
\end{align}
where the last equality follows from Lemma \ref{preservation of dimension}(\ref{preservation of dimension 2}) (by taking $K \coloneqq 2$, $h_1 (n) \coloneqq n^{\alpha-\varepsilon}$, $h_2(n) \coloneqq n$, $h_3(n) \coloneqq n^{\alpha+\varepsilon}$, and $\mathcal{S}(m, h_1, h_2, h_3) \coloneqq \widehat{F}_\varepsilon (m, \alpha)$).

Taking $B \coloneqq \alpha - \varepsilon$, which is strictly greater than $1$, in the proof of Lemma \ref{F infinity}, we see that
\begin{align*}
\widehat{F}_\varepsilon (2, \alpha) 
&= \{ x \in (0,1] : d_n^{\alpha-\varepsilon}(x) < d_{n+1}(x) \leq d_n^{\alpha+\varepsilon}(x) \text{ for all } n \geq 2 \} \\
&\subseteq \{ x \in (0,1] : d_n^{B}(x) < d_{n+1}(x) \text{ for all } n \geq 2 \}
= F_B (2, \infty).
\end{align*}
Hence, by monotonicity of the Hausdorff dimension and \eqref{F infinity F B 2 infinity}, it follows that
\begin{align} \label{F two alpha 2}
\hdim \widehat{F}_\varepsilon (2, \alpha) \leq \hdim F_B(2, \infty) \leq B^{-1} = (\alpha-\varepsilon)^{-1}.
\end{align}
By combining \eqref{F two alpha 1} and \eqref{F two alpha 2}, we deduce that $\hdim F (\alpha) \leq (\alpha-\varepsilon)^{-1}$. Note that $F(\alpha)$ is independent of $\varepsilon$, and so is $\hdim F(\alpha)$. Thus, by letting $\varepsilon \to 0^+$, we conclude that $\hdim F (\alpha) \leq \alpha^{-1}$.
\end{proof}

\begin{proof}[Proof of Theorem \ref{theorem ratio of logarithms}]
The result for the case where $\alpha \in [-\infty, 1)$ follows immediately from the strict increasing condition for digits (Proposition \ref{digit condition lemma}(\ref{digit condition lemma 1})). In fact, for any $x \in \mathbb{I}$, since $d_{n+1} (x) > d_n(x)$ for all $n \in \mathbb{N}$, we have $\liminf_{n \to \infty} \log d_{n+1}(x) / \log d_n(x) \geq 1$. Hence, if $\lim_{n \to \infty} \log d_{n+1}(x) / \log d_n(x) = \alpha$, then $\alpha \geq 1$. Thus, $F(\alpha) = \varnothing$ and $\hdim F(\alpha) = 0$ for any $\alpha \in [-\infty, 1)$.

For each $\alpha \in [1, \infty]$, combine Lemmas \ref{alpha one case}--\ref{F alpha upper bound} to deduce that $\hdim F(\alpha) = \alpha^{-1}$ with the convention $\infty^{-1} = 0$.
\end{proof}

\subsection{The central limit theorem and the law of the iterated logarithm}

This subsection is devoted to proving the results for the sets related to the CLT and the LIL, respectively. We first prove a more general (in some sense) result, Theorem \ref{theorem C psi beta}. 

\begin{proof}[Proof of Theorem \ref{theorem C psi beta}]
We first note that since $\psi(n) n^{-1} \to 0$ as $n \to \infty$ by \eqref{psi n over n and sum of f over n squared} and since $\psi$ is positive real-valued by the hypothesis,
\begin{align} \label{n over psi n}
\lim_{n \to \infty} \frac{n}{\psi (n)} = \infty
\quad \text{and} \quad
\lim_{n \to \infty} \frac{\log ( 1 + \psi (n) n^{-1} )}{\psi (n) n^{-1}} = 1.
\end{align}

We consider three cases according to whether $\beta = \infty$, $\beta = -\infty$, or $\beta \in (-\infty, \infty)$.

{\sc Case I.}
Let $\beta = \infty$. Consider the set $A(e^2)$ defined in \eqref{definition of A alpha} and of full Hausdorff dimension by Corollary \ref{corollary A alpha}. For any $x \in A(e^2)$, since $(d_n(x))^{1/n} \to e^2$ as $n \to \infty$, we find by using \eqref{n over psi n} that
\[
\lim_{n \to \infty} \frac{\log d_n(x)-n}{\psi (n)} = \lim_{n \to \infty} \frac{\log d_n(x)-n}{n} \cdot \frac{n}{\psi(n)} = (2-1) \cdot \infty = \infty.
\]
So $A(e^2) \subseteq C(\psi, \infty)$, and we conclude that $\hdim C(\psi, \infty) = 1$ by monotonicity of the Hausdorff dimension.

{\sc Case II.}
Let $\beta = -\infty$. Consider the set $A(e^{1/2})$ defined in \eqref{definition of A alpha} and of full Hausdorff dimension by Corollary \ref{corollary A alpha}. For any $x \in A(e^{1/2})$, since $(d_n(x))^{1/n} \to e^{1/2}$ as $n \to \infty$, we find by using \eqref{n over psi n} that
\[
\lim_{n \to \infty} \frac{\log d_n(x)-n}{\psi (n)} = \lim_{n \to \infty} \left( \frac{\log d_n(x)-n}{n} \cdot \frac{n}{\psi (n)} \right)  = \left( \frac{1}{2} - 1 \right) \cdot \infty = -\infty.
\]
So $A(e^{1/2}) \subseteq C(\psi, -\infty)$, and we conclude that $\hdim C(\psi, -\infty) = 1$ by monotonicity of the Hausdorff dimension.

{\sc Case III.}
Let $\beta \in (-\infty, \infty)$. Let $f \colon \mathbb{N} \to \mathbb{R}$ be defined by $f(n) \coloneqq n + \beta \psi (n)$ for each $n \in \mathbb{N}$. It follows by Lemma \ref{dimension of E check} that $\hdim \mathcal{E} = 1$, where $\mathcal{E} = \mathcal{E} ((l_n)_{n \in \mathbb{N}}, (r_n)_{n \in \mathbb{N}})$ is defined as in \eqref{definition of E ln rn} and $(l_n)_{n \in \mathbb{N}}$ and $(r_n)_{n \in \mathbb{N}}$ are defined as in \eqref{ln rn definition}. Then, for any $x \in \mathcal{E}$, we have
\[
\beta < \frac{\log d_n(x) - n}{\psi(n)} \leq \beta + \frac{1}{n} \frac{\log ( 1 + \psi (n) n^{-1})}{\psi (n) n^{-1}}
\]
for all sufficiently large $n \in \mathbb{N}$, where the right-hand side converges to $\beta$ as $n \to \infty$ by \eqref{n over psi n}.
Hence, $x \in C(\psi, \beta)$, and this proves that $\mathcal{E} \subseteq C(\psi, \beta)$. Thus $\hdim C(\psi, \beta) = 1$ by monotonicity of the Hausdorff dimension.
\end{proof}

\subsubsection{The CLT}

We divide the proof of Theorem \ref{theorem E alpha beta} into three lemmas according to the value of $\alpha$: $\alpha \in (-\infty, 1)$, $\alpha = 1$, or $\alpha \in (1, \infty)$.

\begin{lemma} \label{alpha less than one}
$\hdim E(\alpha, \beta) = 1$ for all $\alpha \in (-\infty, 1)$ and $\beta \in [ -\infty, \infty ]$.
\end{lemma}

\begin{proof}
Let $\alpha \in (-\infty, 1)$ and $\beta \in [ -\infty, \infty ]$. Let $\psi \colon \mathbb{N} \to (0, \infty)$ be defined by $\psi (n) \coloneqq n^\alpha$ for each $n \in \mathbb{N}$.

\begin{claim}
$\psi (n+1) - \psi (n) \to 0$ as $n \to \infty$.
\end{claim}

\begin{proof} [Proof of Claim] \renewcommand\qedsymbol{$\blacksquare$}
If $\alpha=0$, then the convergence holds trivially since $\psi (n+1) - \psi (n) \equiv 0$. If $\alpha \in (-\infty, 0)$, then $(n+1)^\alpha$ and $n^\alpha$ both converge to $0$ as $n \to \infty$, and hence $\psi (n+1) - \psi (n) \to 0$ as $n \to \infty$. It remains to consider the case where $\alpha \in (0,1)$. Put
\[
g(x) \coloneqq (x+1)^\alpha - x^\alpha = x^\alpha \left[ \left( 1 + \frac{1}{x} \right)^\alpha-1 \right] = \frac{\left( 1 + x^{-1} \right)^\alpha-1}{x^{-\alpha}}
\]
for $x \in (0, \infty)$. Notice that both $\left( 1 + x^{-1} \right)^\alpha-1$ and $x^{-\alpha}$ tend to $0$ as $x \to \infty$. Now, using L'H{\^o}spital's rule, we see that
\begin{align*}
\lim_{x \to \infty} g(x)
&= \lim_{x \to \infty} \frac{\alpha \left( 1 + x^{-1} \right)^{\alpha-1} \left( - x^{-2} \right)}{-\alpha x^{-\alpha-1}}
= \lim_{x \to \infty} \frac{\left( 1 + x^{-1} \right)^{\alpha-1}}{x^{-\alpha+1}} 
= \lim_{x \to \infty} (x+1)^{\alpha-1} 
= 0,
\end{align*}
where the last equality holds since $\alpha - 1 < 0$. Thus $g(n) = \psi (n+1) - \psi (n) \to 0$ as $n \to \infty$, and the claim is proved.
\end{proof}

It is clear that $e^{pn} \psi (n) = e^{pn} n^\alpha \to \infty$ as $n \to \infty$ for any $p>0$. Hence, $\psi \colon \mathbb{N} \to (0, \infty)$ satisfies conditions (\ref{theorem C psi beta 1}) and (\ref{theorem C psi beta 2}) in Theorem \ref{theorem C psi beta}. Now, notice that $E(\alpha, \beta) = C(\psi, \beta)$, where $C(\psi, \beta)$ is defined as in \eqref{definition of C psi beta}. Since $\hdim C(\psi, \beta) = 1$ by Theorem \ref{theorem C psi beta}, we conclude that $\hdim E(\alpha, \beta) = 1$.
\end{proof}

\begin{lemma} \label{alpha one}
For each $\beta \in [-\infty, \infty]$,
\[
\hdim E(1, \beta) =
\begin{cases}
1, &\text{if } \beta \in [-1, \infty]; \\
0, &\text{if } \beta \in [-\infty, -1), \text{ because } E(1, \beta) = \varnothing.
\end{cases}
\]
\end{lemma}

\begin{proof}
This follows from Corollary \ref{corollary A alpha} and the conventions $\infty + 1 = \infty$, $-\infty + 1 = -\infty$, $e^\infty = \infty$, and $e^{-\infty} = 0$. For any $\beta \in [-\infty, \infty]$ and $x \in (0,1]$, note that
\[
\lim_{n \to \infty} (d_n(x))^{1/n} = e^{\beta + 1} \iff 
\lim_{n \to \infty} \frac{\log d_n(x)-n}{n^1} = \beta,
\]
which shows that $A(e^{\beta+1}) = E(1, \beta)$, where $A(e^{\beta+1})$ is defined in \eqref{definition of A alpha}. By Corollary \ref{corollary A alpha}, we have $\hdim A (e^{\beta+1}) = 1$ if $e^{\beta+1} \in [1, \infty]$, i.e., if $\beta \in [-1,\infty]$; $A(e^{\beta+1}) = \varnothing$ otherwise.
\end{proof}

\begin{lemma} \label{alpha between one and infinity}
For all $\alpha \in (1, \infty)$ and $\beta \in [-\infty, \infty]$, we have
\[
\hdim E(\alpha, \beta) = 
\begin{cases}
1, &\text{if } \beta \in [0, \infty]; \\
0, &\text{if } \beta \in [-\infty, 0), \text{ because } E(\alpha, \beta) = \varnothing.
\end{cases}
\]
\end{lemma}

\begin{proof} 
Let $\alpha \in (1, \infty)$. Assume first that $\beta \in (-\infty, 0)$. We prove that $E(\alpha, \beta) = \varnothing$. Let, if possible, $x \in E(\alpha, \beta)$. Then, for $\varepsilon \coloneqq - \beta/2 > 0$, there exists $N \in \mathbb{N}$ such that $n + (\beta-\varepsilon) n^\alpha < \log d_n(x) < n + (\beta + \varepsilon) n^\alpha$ for all $n \geq N$. But then, since $\alpha>1$ and $\beta < 0$, we have $\log d_n(x) < n + (1/2) \beta n^\alpha < 0$ for all sufficiently large $n \in \mathbb{N}$. This contradicts the fact that the $d_n(x)$ are positive integer-valued so that $\log d_n(x) \geq 0$ for any $n \in \mathbb{N}$. Thus $E(\alpha, \beta) = \varnothing$, as desired. A similar argument shows that $E(\alpha, -\infty) = \varnothing$. 

Now, assume $\beta \in (0, \infty)$. Define $\phi (n) \coloneqq \beta n^{\alpha}$ for each $n \in \mathbb{N}$. Clearly, $\phi \colon \mathbb{N} \to (0, \infty)$ is non-decreasing, and $\gamma \coloneqq \lim_{n \to \infty} \phi (n)/\log n = \infty$. Then $E(\phi) \subseteq E(\alpha, \beta)$, where $E(\phi)$ is defined as in \eqref{definition of E phi}, since for any $x \in E(\phi)$,
\[
\lim_{n \to \infty} \frac{\log d_n(x)-n}{n^\alpha} = \lim_{n \to \infty} \left( \frac{\log d_n(x) - n}{\phi (n)} \cdot \beta \right) = (1-0) \cdot \beta = \beta
\]
so that $x \in E(\alpha, \beta)$. Note that since the map $x \mapsto x^\alpha$ is increasing on $(0, \infty)$, we have
\begin{align*}
0 \leq \xi 
\coloneqq \limsup_{n \to \infty} \frac{\phi(n+1)}{\sum \limits_{k=1}^n \phi (k)} 
= \limsup_{n \to \infty} \frac{\beta (n+1)^\alpha}{\sum \limits_{k=1}^n (\beta k^\alpha)} 
\leq \lim_{n \to \infty} \frac{(n+1)^\alpha}{\displaystyle \int_0^n x^\alpha \, dx}
= \lim_{n \to \infty} \frac{(n+1)^\alpha}{\dfrac{1}{\alpha+1} n^{\alpha+1}}
=0,
\end{align*}
and hence $\xi = 0$. By Theorem \ref{theorem E phi}, $\hdim E( \phi ) = (1+\xi)^{-1} = 1$, and thus $\hdim E(\alpha, \beta) = 1$ by monotonicity of the Hausdorff dimension. 

Take $\alpha' \in (1, \alpha)$ and $\beta' \in (0, \infty)$. Then $E (\alpha', \beta') \subseteq E(\alpha, 0)$, since for any $x \in E(\alpha', \beta')$,
\[
\lim_{n \to \infty} \frac{\log d_n(x)-n}{n^\alpha} = \lim_{n \to \infty} \frac{\log d_n(x)-n}{n^{\alpha'}} \cdot \frac{1}{n^{\alpha-\alpha'}} = \beta' \cdot 0 = 0.
\]
Since $\hdim E(\alpha', \beta') = 1$ by the preceding paragraph, we get $\hdim E(\alpha, 0) = 1$ by monotonicity of the Hausdorff dimension.

Take $\alpha'' \in (\alpha, \infty)$ and $\beta'' \in (0, \infty)$. Then $E (\alpha'', \beta'') \subseteq E(\alpha, \infty)$, since for any $x \in E(\alpha'', \beta'')$,
\[
\lim_{n \to \infty} \frac{\log d_n(x)-n}{n^\alpha} = \lim_{n \to \infty} \frac{\log d_n(x)-n}{n^{\alpha''}} \cdot n^{\alpha''-\alpha} = \beta'' \cdot \infty = \infty.
\]
Since $\hdim E(\alpha'', \beta'') = 1$ by the second paragraph, it follows that $\hdim E(\alpha, \infty) = 1$ by monotonicity of the Hausdorff dimension.
\end{proof}

\begin{proof} [Proof of Theorem \ref{theorem E alpha beta}]
Combine Lemmas \ref{alpha less than one}--\ref{alpha between one and infinity}.
\end{proof}

\subsubsection{The LIL}

We prove Corollary \ref{corollary L beta}, the proof of which is similar to that of Lemma \ref{alpha less than one}.

\begin{proof}[Proof of Corollary \ref{corollary L beta}]
Let $\psi \colon \mathbb{N} \to (0, \infty)$ be defined by 
\[
\psi (n) \coloneqq 
\begin{cases}
n, &\text{if } n \in \{ 1, 2 \}; \\
\sqrt{2n \log \log n}, &\text{if $n \geq 3$}.
\end{cases}
\]

\begin{claim} \label{difference claim 2}
$\psi (n+1) - \psi (n) \to 0$ as $n \to \infty$.
\end{claim}

\begin{proof} [Proof of Claim] \renewcommand\qedsymbol{$\blacksquare$}
For $n \geq 3$, write
\begin{align*}
\frac{\psi (n+1) - \psi (n)}{\sqrt{2}} 
&= \frac{(n+1) \log \log (n+1) - n \log \log n}{\sqrt{(n+1) \log \log (n+1)} + \sqrt{n \log \log n}} \\
&= \frac{n [\log \log (n+1) - \log \log n] + \log \log (n+1)}{\sqrt{(n+1) \log \log (n+1)} + \sqrt{n \log \log n}}.
\end{align*}
Observe that $\sqrt{(n+1) \log \log (n+1)} + \sqrt{n \log \log n} \to \infty$ as $n \to \infty$ and that
\begin{align*}
0 
&\leq \frac{\log \log (n+1)}{\sqrt{(n+1) \log \log (n+1)} + \sqrt{n \log \log n}} \\
&\leq \frac{\log \log (n+1)}{\sqrt{(n+1) \log \log (n+1)}}
= \sqrt{\frac{\log \log (n+1)}{n+1}} 
\to 0
\end{align*}
as $n \to \infty$; hence, it suffices to show that $n [\log \log (n+1) - \log \log n] \to 0$ as $n \to \infty$. First note that 
\begin{align} \label{interim convergence}
\frac{x}{\log x} - \frac{x}{\log (x+1)} 
= \frac{x \log (x+1) - x \log x}{(\log x) (\log (x+ 1))} 
= \frac{\log \left( 1 + x^{-1} \right)^x}{(\log x) (\log (x+ 1))}  
\to \frac{1}{\infty} = 0
\end{align}
as $x \to \infty$. Put 
\[
g(x) \coloneqq x [\log \log (x+1) - \log \log x] = \frac{\log \log (x+1) - \log \log x}{x^{-1}}
\]
for $x \in (1, \infty)$. Notice that both $\log \log (x+1) - \log \log x = \log \left( {\log (x+1)}/{\log x} \right)$ and $x^{-1}$ tend to $0$ as $x \to \infty$. Now, using L'H{\^o}spital's rule, we see that
\begin{align*}
\lim_{x \to \infty} g(x)
&= \lim_{x \to \infty} \frac{[(x+1) \log (x+1)]^{-1} - (x \log x)^{-1}}{- x^{-2}} 
= \lim_{x \to \infty} \left( \frac{x}{\log x} - \frac{x^2}{(x+1) \log (x+1)} \right) \\
&= \lim_{x \to \infty} \left[ \frac{x}{x+1} \left( \frac{1}{\log x} + \frac{x}{\log x} - \frac{x}{\log (x+1)} \right) \right] 
= 1 (0+0) = 0,
\end{align*}
where we have used \eqref{interim convergence} for the second-to-last equality. Thus, $g(n) = n [\log \log (n+1) - \log \log n] \to 0$ as $n \to \infty$, and the claim is proved.
\end{proof}

It is clear that $e^{pn} \psi (n) \to \infty$ as $n \to \infty$ for any $p>0$. Hence, $\psi \colon \mathbb{N} \to (0, \infty)$ satisfies conditions (\ref{theorem C psi beta 1}) and (\ref{theorem C psi beta 2}) in Theorem \ref{theorem C psi beta}. Now, notice that $L(\beta) = C(\psi, \beta)$, where $C(\psi, \beta)$ is defined as in \eqref{definition of C psi beta}. Since $\hdim C(\psi, \beta) = 1$ by Theorem \ref{theorem C psi beta}, we conclude that $\hdim L(\beta) = 1$.
\end{proof}

\begin{proof} [Proof of Corollary \ref{corollary LIL exceptional set}]
Consider the set $L(0)$ defined in \eqref{definition of L beta}. Clearly, $L(0)$ is a subset of the set on which \eqref{law of the iterated logarithm} fails. This is because for any $x$ where \eqref{law of the iterated logarithm} holds, the limit
\[
\lim_{n \to \infty} (\log d_n(x)-n)/\sqrt{2n \log \log n}
\]
does not exist. But $\hdim L(0) = 1$ by Corollary \ref{corollary L beta}, and the result follows by monotonicity of the Hausdorff dimension.
\end{proof}

\begin{remark} \label{remark to proof of corollary LIL exceptional set}
As mentioned in Remark \ref{remark to corollary LIL exceptional set}, we show that Theorem \ref{theorem A kappa} implies Corollary \ref{corollary LIL exceptional set}. To this end, let $\kappa \in (1, \infty)$, and consider the set $A_\kappa$ defined in \eqref{definition of A kappa}. Then for any $x \in A_\kappa$, we have $\log d_n(x)-n \geq (\kappa-1)n > 0$ for all $n \in \mathbb{N}$, so that $(\log d_n(x)-n)/\sqrt{2n \log \log n} \to \infty$ as $n \to \infty$. This shows that no $x \in A_\kappa$ satisfies \eqref{law of the iterated logarithm}, i.e., $A_\kappa$ is a set of exceptions to the LIL. But $\hdim A_\kappa = 1$ by Theorem \ref{theorem A kappa}, and therefore the monotonicity argument yields Corollary \ref{corollary LIL exceptional set}.
\end{remark}

\section*{Acknowledgements}

I am grateful to my advisor, Dr.\ Hanfeng Li, for his very useful comments and suggestions, constant support, and encouragement to work on the subject.

\end{document}